\documentclass[reqno]{amsart}
\usepackage{amsmath,amsbsy,amsfonts}
\usepackage{color}

\usepackage[utf8]{inputenc}
\usepackage{graphicx}
\usepackage{amsmath,latexsym,amssymb}
\newtheorem{theorem}{Theorem}
\newtheorem{proposition}{Proposition}[section]

\newtheorem{definition}[proposition]{Definition}
\newtheorem{lemma}[proposition]{Lemma}
\newtheorem{remark}[proposition]{Remark}

\newcommand{\dd}{\mathrm{d}}

\numberwithin{equation}{section}


\usepackage{tikz}

\usepackage{color}
\numberwithin{equation}{section}

\def\R{\mathbb{R}}

\title[Critical fractional  elliptic equations with exponential growth ]{Critical fractional  elliptic equations with exponential growth without Ambrosetti-Rabinowitz type condition}

\author[H. P. Bueno]{Hamilton P. Bueno}\thanks{First author takes part in the project 422806/2018-8 by CNPq/Brazil}
\author[E.Huerto Caqui]{Eduardo Huerto Caqui}\thanks{Second author was supported by CAPES/Brazil}
\author[O. H. Miyagaki]{Olimpio H. Miyagaki } \thanks{Third author was supported by Grant 2019/24901-3 by S\~ao Paulo Research Foundation (FAPESP) and Grant 307061/2018-3 by CNPq/Brazil.} 

\address[H. P. Bueno]{ Department of Matematics, Universidade Federal de  Minas Gerais, 31270-901 - Belo Horizonte - MG, Brazil}
	\email{hamilton@mat.ufmg.br}
\address[E. Huerto Caqui]{ Department of Matematics, Universidade Federal de  Minas Gerais, 31270-901 - Belo Horizonte - MG, Brazil}
	\email{analisis\_11@hotmail.com}
\address[O.H. Miyagaki]{
Department of Mathematics, 
 Universidade Federal de S\~ao Carlos, 13565-905, S\~ao Carlos-SP, Brazil}
\email{olimpio@ufscar.br, ohmiyagaki@gmail.com}



\begin{document}

\begin{abstract}
 
In this paper we establish, using variational methods combined with the Moser-Trudinger inequality, existence and multiplicity of weak solutions for a class of critical fractional elliptic equations with exponential growth without a Ambrosetti-Rabinowitz-type condition. The interaction of the  nonlinearities with the spectrum of the fractional  operator will used to study the existence and multiplicity of solutions. The main technical result proves that a local minimum in $C_{s}^0(\overline{\Omega})$ is also a local minimum in $W^{s,p}_0$ for nonlinearities with exponential growth.
\end{abstract}

\maketitle

\section{Introduction}

In this paper we consider existence and multiplicity of solutions to the Dirichlet problem 
\begin{align}\label{principal2}
 \left\{
\begin{array}{rlll}
(-\Delta)^{s}_{p} u&=& -\lambda\vert u\vert^{q-2}u+a\vert u\vert^{p-2}u+f(u) &\textrm{in}\ \Omega,\\
u&=&0 &\textrm{in}\ \mathbb{R}^N\setminus\Omega,
\end{array}
\right.
\end{align}
where $ (-\Delta)^{s}_{p} $ is the fractional $p$-Laplacian, $\Omega\subset\mathbb{R}^{N}$ is a bounded smooth domain, $\lambda>0$ and $a\in\mathbb{R}$ are parameters, $N=sp,$ and $0<s<1<q<2\leq p$. Here
$$(-\Delta)^{s}_pu(x)=2\lim_{\epsilon\to 0}\int_{\R^N\setminus B(x,\epsilon)}\frac{\vert u(x)-u(y)\vert^{p-2}(u(x)-u(y))}{\vert x-y\vert^{N+sp}}\dd y,$$ 
where  $u$ is a measurable function and $x\in\R^{N}.$

We suppose that the nonlinearity $f$ has exponential growth, both critical and subcritical in the Trudinger-Moser sense. 

Furthermore, we consider the particular case  
$sp=N=1$, $p=2$ and $s=1/2$, that is, we study the problem  
\begin{align}\label{C1M_1}
\left\{
\begin{array}{rlll}
(-\Delta)^{1/2} u&=&-\lambda\vert u\vert^{q-2}u+au+ f(u) &\textrm{in}\;\;(0,1),\\
u&=&0 &\textrm{in}\;\;\mathbb{R}\setminus(0,1).
\end{array}
\right.
\end{align}

Recently, non-local problems  have been extensively studied in the literature and have attracted the attention of many mathematicians from different fields of research. This type of non-local problems appears in the description of various phenomena in the applied sciences, such as optimization, finance, phase transitions,  material science and water waves, image processing, etc. See the excellent book by Caffarelli on this subject \cite{Caffarelli}, but also an elementary introduction to this topic by Di Nezza, Palatucci and Valdinoci \cite {dinezza}.

Em 1994,  Ambrosetti, Brezis and  Cerami \cite{Amb} established existence and multiplicity of solution for a local  problem involving concave-convex nonlinearities and Sobolev critical exponent, namely, $2^*=\frac{2N}{N-2}(N \geq 3).$

After this work, there has been a growing interest in the study of multiplicity of solutions for local problems of the type $$-\Delta u=\mu\vert u\vert^{q-2}u+g(u)\quad \mbox{in}\quad\Omega,$$ when  $g$ is asymptotically linear and asymmetric, that is, $g$ satisfies the  Ambrosetti-Prodi-type condition  given by (see \cite{deFiguei})
$g_-=\displaystyle\lim_{t\to-\infty}\frac{g(t)}{t}<\lambda_k<g_+=\displaystyle\lim_{t\to+\infty}\frac{g(t)}{t},$  where $\{\lambda_{k}\}_{k\geq 1}$ denotes the sequence of eigenvalues of $(-\Delta)$ considered in  $H_0^1(\Omega).$  In Chabrowsky and Yang \cite{cha} a problem with Neumann boundary condition was considered, while in Motreanu, Motreanu and Papageorgiou \cite{mont} the authors study a problem involving a local $p$-Laplacian. In  \cite{paiva}, de Paiva and Massa studied the local problem
\begin{equation}\label{Diric} 
\left\{\begin{array}{rcll}
-\Delta u&=&-\lambda\vert u\vert^{q-2}u+au +g(u)&\text{in }\ \Omega,\\
u&=&0&\text{on }\ \partial\Omega,
\end{array}
\right.
\end{equation}
with $1<q<2,$ $\lambda>0,$ $a\in[\lambda_{k},\lambda_{k+1}),$ and the nonlinearity $g$ satisfying subcritical polynomial growth at infinity, among other conditions, while the critical case was considered in de Paiva and Presoto \cite{paiva1}, where three solutions  for problem \eqref{Diric} were obtained: a positive, a negative and a sign changing solution. Problem \eqref{Diric} with critical polynomial growth was handled by Miyagaki, Motreanu and Pereira \cite{O} for the fractional Laplacian operator. To complete our references, we would like to cite some papers. For instance,  \cite{Ambro, Amb, Colorado,Perera} for concave problems,  \cite{Autuori,barrios,barrio,brandle} for problems involving the fractional Laplacian and, for the fractional $p$-Laplacian, we cite \cite{bartolo,HUXIA, HUANG,mosconi,Che}. See also references therein.

With respect to nonlinearities with exponential growth for a problem like \eqref{principal2}, in the limit case $N=sp$, Bahrouni \cite{ANOUR} proved a version of the Trudinger-Moser inequality for fractional spaces, which was improved by Takahashi \cite{TAKAHASHI}, who obtained, among other things, optimality of the upper bound. For local elliptic problems with exponential growth  nonlinearity we would like to cite, e.g., \cite{de Figueiredo,deFigueiredo,Yan} and  references therein. 

The pioneering paper for fractional Laplacian, by Iannizzotto and Squassina  \cite{ianni} considered a nonlinearity with exponential growth, but it was proved by de Figueiredo, Miyagaki and Ruf \cite[p.142]{deFigueiredo} that the Ambrosetti-Rabinowitz (AR) condition was satisfied in \cite{ianni}. Namely, the (AR) condition is fulfilled if there exist $\mu>p$ and $R>0$ such that $$0<\mu F(t)\leq f(t)t,\;\mbox{ for all }\,  \vert t\vert\geq R,\;\mbox{ where} \,F(t)=\displaystyle\int_{0}^tf(s)\dd s \leqno{(AR)}$$ 
and in this situation, 
\begin{equation*}
\displaystyle{\lim_{\vert t\vert\to+\infty}}\frac{F(t)}{\vert t\vert^p}=+\infty
\end{equation*}
follows immediately from (AR). The main role of (AR)  is to guarantee that Palais-Smale sequences are bounded. Many authors have been working to drop this condition for problem with polynomial growth, e.g., \cite{Costa, Jeanjean,Liu, Li,S.B.,Schechter} and references therein. For exponential growth without the (AR) condition  we cite, for instance,  \cite{NGYEN,NGYE}. 
Recently,  Pei \cite{Ruichang} proved a existence result for a superlinear $p$-fractional problem with exponential growth.
 
Motivated by \cite{Ruichang} and \cite{paiva1},
in this work we obtain results of existence and multiplicity of solutions for \eqref{principal2}.

We look for solutions to  \eqref{principal2} in the Sobolev space
 $$W^{s,p}(\mathbb{R}^{N}):=\left\lbrace u\in L^p(\mathbb{R}^{N}) : \int_{\mathbb{R}^{2N}}\frac{\vert u(x)-u(y)\vert^p}{|x-y|^{N+sp}}\dd x\dd y<\infty\right\rbrace.$$
Since solutions must be equal $0$ outside $\Omega$, it is natural to consider $X_p^s\subset W^{s,p}(\mathbb{R}^{N})$ given by
 $$X_p^s=\left\{u\in W^{s,p}(\mathbb{R}^{N}) : u = 0 \textrm{ on } \mathbb{R}^N\setminus\Omega\right\}.$$
We denote by $\|\cdot\|_{X_p^s}$ the norm in $X_p^s$ (see Section \ref{prel}) and
\begin{equation*}
\lambda^{*}=\inf\left\{\|u\|_{X_p^s}^{p}\,:\,u\in W, \ \|u\|_{L^{p}(\Omega)}^{p}=1\right\},
\end{equation*}
where 
\begin{equation*}
W=\bigg\{ u\in X_p^s\; :\; \langle A(\varphi_1),u\rangle=0\bigg\},
\end{equation*} with $A:X_p^s\to (X_{p}^{s})^*$ defined, for all $u,v\in X_p^s$, by
\begin{equation}\label{operador}
\langle A(u),v\rangle=\displaystyle{\int_{\mathbb{R}^{2N}}\frac{\vert u(x)-u(y)\vert^{p-2}(u(x)-u(y))(v(x)-v(y))}{\vert x-y\vert^{N+sp}}}\dd x\dd y.
\end{equation}

We will prove that $X_p^s=W\oplus\textrm{span}\{\varphi_{1}\}$ if $\lambda_1<\lambda^{*}$, where $\varphi_{1}$ is the (positive, $L^p$-normalized) autofunction associated with  the first eigenvalue $$\lambda_1=\inf\left\{ [u]^p_{ W^{s,p}(\mathbb{R}^N)}\,:\, u\in X_p^s, \ \|u\|_{L^p(\Omega)}=1\right\}$$ of $(-\Delta)_p^s$ in the space $X_p^s$, see Section \ref{prel} for notation.

To cope with nonlinearities involving exponential growth, the main tool is the following inequality, known as  ``Moser-Trudinger inequality" in the literature.  We will make use of the following version of the Moser-Trudinger inequality, based on \cite[Lema 2.5]{ANOUR}.
\begin{proposition}\label{C2P_4.4}
Suppose that $0<s<1$, $p\geq 2$ and $N=sp$. Then there exists $\alpha_{s,N}^{*}=\alpha(s,N)$ such that, for all $0\leq\alpha<\alpha_{s,N}^{*}$, 
\begin{equation*}
\int_{\Omega}\exp\left(\alpha\vert u\vert^{\frac{N}{N-s}}\right)\dd x\leq H_{\alpha},
\end{equation*}
for all $u\in X_p^s$ such that  $\|u\|_{X_p^s}\leq 1$, where $H_\alpha>0$ is a constant.
\end{proposition}

An adequate version of Proposition \ref{C2P_4.4} in the special case $p=2$, $s=1/2$ and $N=1$ is given in the sequence (see \cite[Teorema 1]{TAKAHASHI} and \cite[Proposição 1.1]{martinazzi}).
\begin{proposition}
There exists $K>0$ so that
$$\sup_{u\in X, \|u\|_X\leq 1}\int_{0}^1\exp\left(\alpha\vert u\vert^{2}\right)\dd x\leq K,\;\textrm{ for all }\,\alpha\leq\pi,$$
where $X$ stands for $X^{1/2}_2$ and $\|\cdot\|_X$ denotes its norm.

The inequality is optimal if $\alpha>\pi$, since the left-hand side of the inequality is then equal to $\infty$. 
\end{proposition}

Considering \eqref{principal2} in the case of subcritical exponential growth in the Trudinger-Moser sense, we suppose that $f$ satisfies
\begin{enumerate}
\item [$(f_{1,p})$] $f\in C(\mathbb{R},\mathbb{R})$, $f(0)=0$ and $F(t)\geq 0$ for all $t\in\mathbb{R}$, where  $F(t)=\displaystyle{\int_0^t}f(s)\dd s$;
 \item [$(f_{2,p})$] $\displaystyle\lim_{\vert t\vert\to \infty}\frac{\vert f(t)\vert}{\exp(\alpha\vert t\vert^{\frac{N}{N-s}})}=0$, for all $\alpha>0$;
\item [$(f_{3,p})$] $\displaystyle\lim_{\vert t\vert\to 0}\frac{  f(t)}{\vert t\vert^{p-2}t}=0$;
\item [$(f_{4,p})$] $\displaystyle\lim_{\vert t\vert\to \infty}\frac{ F(t)}{\vert t\vert^p}=+\infty$. 
\end{enumerate}

In the case of a critical exponential growth, we change $(f_{2,p})$ for
\begin{enumerate}
\item [$(f'_{2,p})$] there exists $\alpha_0>0$ such that 
 $$\displaystyle\lim_{\vert t\vert\to \infty}\frac{ \vert f(t)\vert}{\exp(\alpha \vert t\vert^{\frac{N}{N-s}})}= \left\{ \begin{array}{rc}
\infty, &\quad\;\textrm{if}\quad 0<\alpha<\alpha_0 \\
0, &\textrm{if}\quad \alpha>\alpha_0.
\end{array} \right.$$
\end{enumerate}
Keeping up with the conditions  $(f_{1,p})$ and $(f_{3,p})$, we suppose additionally that $f$ satisfies
\begin{enumerate}
\item [$(f_{5,p})$]  $\displaystyle\frac{f(t)}{\vert t\vert^{p-2}t}\ \textrm{ is increasing if }\ t>0,\ \textrm{ and decreasing if}\ t< 0$;
\item [$(f_{6,p})$] For all sequence $(u_n)\subset X_p^s$, if
 $$\left\{ \begin{array}{llll}
u_n&\rightharpoonup &u, &\textrm{in}\quad X_p^s, \\
f(u_n)&\to &f(u), &\textrm{in}\quad  L^{1}(\Omega),
\end{array} \right.$$
then $F(u_n)\to F(u)\;$ in $\;L^{1}(\Omega)$;
\item [$(f_{7,p})$] There exist $r>p$ and $C_{r}>0$ such that 
$F(t)\geq \dfrac{C_{r}}{r}\vert t\vert^{r},\;\;\textrm{ for all }\; t\in\mathbb{R},$
verifying
$$C_r>\displaystyle\left[2\frac{N}{s}\left(\frac{\alpha_0}{\alpha_{s,N}^*}\right)^{\frac{N-s}{s}}\frac{(r-p)}{pr}\right]^{      \frac{r-p}{p}}\frac{1}{C},$$
where  $\alpha_{s,N}^*$  is the constant given in Proposition  \ref{C2P_4.4} and $$C=\inf_{u\in \mathbb{F} }\frac{\|u\|_{L^r}}{\|u\|_{X^{s}_{p}}},$$
where $\mathbb{F}=\textrm{span}\{ \varphi_1,\varphi\}$ for $\varphi\in W$. 
\end{enumerate} 
\begin{remark}Condition $(f_{6,p})$ was supposed by  \cite{NGYEN}, \cite{NGYE} and \cite{Ruichang} in the case $u=0$.
Observe that $(f_{7,p})$ implies $(f_{4,p})$.

Hypotheses $(f_{1,p})-(f_{4,p})$ are satisfied by $f(t)=\vert t\vert^{p-2}t\log(1+\vert t\vert)$, a function that does not verify the  $(AR)$ condition. 

On its turn, considering $0<\sigma <1$, the function
$$f(t) =\left\{\begin{array}{ll}
\sigma t^{r-1}+C_{r}t^{r-1},
&\textrm{ if } 0\leq t\displaystyle\leq (p-1)^{\frac{N-s}{N}},\\
t^{\frac{N}{N-s}}\exp\left(t^\frac{N}{N-s}-(p-1)\right)+C_rt^{r-1}\\+\sigma\displaystyle (p-1)^{\frac{N-s}{N}(r-1)}-(p-1)^{\frac{s}{N}}, 
&\textrm{ if } t> (p-1)^{\frac{N-s}{N}}
\end{array}\right.$$
satisfies our hypotheses in the critical growth case, if $f(t)=-f(-t)$, for $t<0$. 
\end{remark}

\begin{theorem}\label{0}
Let $\Phi:X_p^s\to\mathbb{R}$ be the  $C^1(X_p^s,\mathbb{R})$ functional defined by $$\Phi(u)=\dfrac{1}{p}\|u\|_{X_p^s}^p-\displaystyle\int_{\Omega} G(u)\dd x,$$ 
where $G(t)=\displaystyle\int_{0}^{t}g(s)\dd s$. 

Let us suppose that $g$ satisfies $(f_{2,p})$ or $(f'_{2,p})$ and that $0$ is a local minimum of $\Phi$ in $C_{s}^0(\overline{\Omega})$, that is, there exists $r_1>0$ such that
\begin{equation}\label{PH1}
\Phi(0)\leq\Phi(z),\;\forall\;z\in X_p^s\cap C_{s}^0(\overline{\Omega}),\;\|z\|_{0,s}\leq r_1.
\end{equation}
	
Then $0$ is a local minimum of $\Phi$ in $X_p^s$, that is, there exists $r_2>0$ such that $$\Phi(0)\leq\Phi(z),\;\forall\; z\in X_p^s,\;\|z\|_{X_p^s}\leq r_2.$$
\end{theorem}
\noindent (See definition of $C_{s}^0(\overline{\Omega})$ in Section \ref{M}.)

Theorem \ref{0} will play an essential role to obtain the next results.
\begin{theorem}[subcritical case]\label{3}
If $\lambda_1 \leq a<\lambda^{*}$ and if $f$ satisfies conditions $(f_{1,p})-(f_{4,p})$ then, for $\lambda$ small enough, problem \eqref{principal2} has at least three nontrivial solutions. Additionally, if $f$ is odd, then \eqref{principal2} has infinitely many solutions.
\end{theorem}
\begin{theorem}[critical case]\label{4}
If $f$ satisfies conditions $(f_{1,p})$, $(f'_{2,p})$, $(f_{3,p})$ and  $(f_{5,p})-(f_{7,p})$ then, for $\lambda$ small enough, problem \eqref{principal2} has at least three nontrivial solutions in the case $\lambda_1 \leq a<\lambda^{*}$.
\end{theorem}

The version of Theorem \ref{0} for the case $N=1$, $p=2$ and $s=1/2$ is given by (see notation introduced before)
\begin{theorem}\label{C2_24}
Let $\Phi:X\to\mathbb{R}$ the $C^1(X,\mathbb{R})$-functional given by $$\Phi(u)=\dfrac{1}{2}\|u\|^2-\displaystyle\int_0^1 G(u)\dd x,$$ 
where $G(t)=\displaystyle\int_{0}^{t}g(s)\dd s$ and $g$ satisfies $(f_{2,2})$ or $(f'_{2,2})$.

Suppose that $0$  is a local minimum of $\Phi$ in $C_{\delta}^0([0,1])$, that is, there exists $r_1>0$ such that 
\begin{equation*}
\Phi(0)\leq\Phi(z),\;\forall\;z\in X\cap C_{\delta}^0([0,1]),\;\|z\|_{0,\delta}\leq r_1,	\end{equation*}
then $0$ is a local minimum of $\Phi$ in $X$, that is, there exists $r_2>0$ such that $$\Phi(0)\leq\Phi(z),\;\forall\; z\in X,\;\|z\|_{X}\leq r_2.$$
\end{theorem}

Considering the eigenvalue sequence  $\{\lambda_{j}\}_{j\geq 1}$ of 
$(-\Delta)^{1/2}$ in $X_2^{1/2}$, we will then prove the following results.
\begin{theorem}[subcritical case]\label{5}
If $\lambda_{k} \leq a<\lambda_{k+1}$ for some $k\in\mathbb{N}\;(k\geq 1)$ and if $f$ satisfies conditions $(f_{1,2})-(f_{4,2})$ then, for $\lambda$ small enough, \eqref{C1M_1} has at least three nontrivial solutions. Additionally, if $f$ is odd, then \eqref{C1M_1} has infinitely many solutions.
\end{theorem}
\begin{theorem}[critical case]\label{6}
If $\lambda_{k} \leq a<\lambda_{k+1}$ for some $k\in\mathbb{N}\;(k\geq 1)$ and if $f$ satisfies $(f_{1,2})$,  $(f'_{2,2})$, $(f_{3,2})$ and $(f_{5,2})-(f_{7,2})$ then, for $\lambda$ small enough, \eqref{C1M_1} has at least three nontrivial solutions.
\end{theorem}

The main achievement of this paper are the minimization results that will be presented in Section \ref{M} (see notation there): we prove that a local minimum in $C_{s}^0(\overline{\Omega})$ is also a local minimum in $W^{s,p}_0$ for nonlinearities with exponential growth. They are the counterpart of the result obtained by de Paiva and Massa \cite{paiva} (also de Paiva and Presoto \cite{paiva1}) and their proofs are obtained by applying ideas developed by Barrios, Colorado, de Pablo and Sanchez \cite{barrios}, Giacomoni, Prashanth and Sreenadh \cite{giacomoni} and Iannizzoto, Mosconi and Squassina \cite{iannizzoto2}. We would like to  emphasize that with exception  of \cite{giacomoni}, which deals with local  $N$-Laplacian case with exponential growth, other references treated  local or non-local Laplacian with polynomial growth.
\section{Preliminaries}\label{prel}

\begin{definition}
We say that $u\in X_{p}^{s}$ is a weak solution to \eqref{principal2} if 
$$\langle A(u),v\rangle=-\lambda\int_{\Omega}\vert u\vert^{q-2}uv\dd x+a\int_{\Omega}\vert u\vert^{p-2}uv\dd x+\int_{\Omega}f(u)v\dd x,$$ 
\end{definition}
for all $v\in X_p^s$, with $A:X_p^s\to (X_{p}^{s})^*$ being defined by \eqref{operador}.

If $sp=N=1$, $p=2$ and $s=1/2$, the operator $A$ defines an inner product in the space $X_2^{1/2}$, which will be denoted by $\langle A(u),v\rangle=\langle u,v\rangle_{X_2^{1/2}}$. This justifies the following definition.
\begin{definition}
We say that $u\in X=X_2^{1/2}$ is a weak solution to \eqref{C1M_1} if
$$\langle u, v\rangle_{X_2^{1/2}}=-\lambda\int_0^1\vert u\vert^{q-2}uv\dd x+ a\int_0^1uv\dd x+\int_{0}^{1}f(u)v\dd x,\; \textrm{ for all }\; v\in X_2^{1/2}.$$
\end{definition}
We recall that  
$$W^{s,p}(\mathbb{R}^{N}):=\left\lbrace u\in L^p(\mathbb{R}^{N}) : \int_{\mathbb{R}^{2N}}\frac{\vert u(x)-u(y)\vert^p}{|x-y|^{N+sp}}\dd x\dd y<\infty\right\rbrace,$$
is a uniformly convex Banach space with the norm
\begin{equation*}
\lVert u\rVert_{ W^{s,p}(\mathbb{R}^N)}:=\left(\Vert u\Vert_{L^p(\mathbb{R}^N)}^p +  [u]_{ W^{s,p}(\mathbb{R}^N)}^p\right)^{1/p},
\end{equation*}
where
$$ [u]_{ W^{s,p}(\mathbb{R}^N)}:=\left(\int_{ \mathbb{R}^{2N}} \frac{\vert u(x)-u(y)\vert^p}{|x-y|^{N+sp}}\dd x\dd y\right)^{1/p}$$
is the Gagliardo seminorm. Since $\Omega\subset\mathbb{R}^{N}$ is a bounded, smooth domain and $0<s<1<p$, 
$[u]_{W^{s,p}(\mathbb{R}^N)}$ defines an equivalent norm in $X_p^s$, see \cite[p.4]{iannizzoto0}. From now on we will consider $W^{s,p}(\mathbb{R}^N)$ equipped with this norm. We also recall (see \cite[Teorema 6.5, 7.1]{dinezza}) that $X_p^s$ is compactly immersed in $L^{r}(\Omega)$ for all $1\leq r<\infty$, the immersion being continuous in the case $r=\infty$.

If $p=2$, $s=1/2$ and $N=1$ we will denote $W^{1/2,2}(\mathbb{R}):=H^{1/2}(\mathbb{R})$ and $X_{2}^{1/2}$ simply by $X$. That is, $$X=\{u\in H^{1/2}(\mathbb{R}) : u = 0 \textrm{ in } \mathbb{R}\setminus (0,1)\}.$$ Of course, $H^{1/2}(\mathbb{R})$ is a Hilbert space with the inner product
$$\langle u,v\rangle_{X}=\int_{\mathbb{R}^2}\frac{(u(x)-u(y))(v(x)-v(y))}{|x-y|^{2}}\dd x\dd y.$$ 

We define the functional  $I_{\lambda,p}:X_p^s\to\mathbb{R}$ by 
\begin{equation*}
I_{\lambda,p}(u)=\displaystyle\frac{1}{p}\|u\|_{X_p^s}^p+\frac{\lambda}{q}\int_{\Omega}\vert u\vert^{q}\dd x-\frac{a}{p}\int_{\Omega}\vert u\vert^p\dd x-\int_{\Omega} F(u)\dd x.
\end{equation*}
In the case $p=2$, $s=1/2$ and $N=1$ we simplify the notation: $I_{\lambda}=I_{\lambda,p}$. Note that, if $f$ is odd, then $I_{\lambda,p}$ is even.

The next result is a direct consequence of  \cite[Proposição 1.3.]{per}.
\begin{lemma}\label{C2P_4.6}
If $u_n\rightharpoonup u$ in $X_p^s$ and $\langle A(u_n),u_n-u\rangle\to 0$, then $u_n\to u$ in $X_p^s$.  
\end{lemma}

Let us consider the Dirichlet problem 
\begin{align}\label{a.2}
\left\{\begin{array}{rlll}
(-\Delta)^{s}_{p} u&=&f(u) &\textrm{in}\ \Omega,\\
u&=&0 &\textrm{in}\ \mathbb{R}^N\setminus\Omega,
\end{array}
\right.
\end{align}
where $\Omega\subset\mathbb{R}^N$ ($N>1$) is a bounded, smooth domain, $s\in(0,1)$, $p>1$ and $f\in L^{\infty}(\Omega)$.

The next two results can be found in Iannizzotto, Mosconi and Squassina \cite{iannizzoto2}, Theorems 1.1 and 4.4, respectively.
\begin{proposition}\label{a.3}
There exist $\alpha\in(0,s]$ and $C_{\Omega}>0$ depending only on $N$, $p$, $s$, with $C_{\Omega}$ also depending on $\Omega$, such that, for all weak solution $u\in X^s_p$ of \eqref{a.2}, $u\in C^\alpha(\overline{\Omega})$ and
$$\|u\|_{C^{\alpha}(\overline{\Omega})} \leq C_{\Omega}\|f\|_{L^{\infty}(\Omega)}^{\frac{1}{p-1}}.$$
\end{proposition} 
\begin{proposition}\label{a.4}
Let $u\in X^s_p$  satisfies $\left|(-\Delta)_{p}^{s} u\right| \leq K$ weakly in $\Omega$ for some $K>0$. Then 
$$|u| \leq\left(C_{\Omega} K\right)^{\frac{1}{p-1}} \delta^{s}\quad a.e.\,\mbox{ in }\Omega,$$ 
for some $C_{\Omega}=C(N, p, s, \Omega)$. 
\end{proposition}
By adapting arguments of Zhang and Shen \cite[Lemma 2]{zhang} we obtain the following result.
\begin{lemma}[Critical and subcritical cases]\label{C2_9}

If $f$ satisfies $(f_{1,p})$, $(f_{2,p})$ (or $(f'_{2,p})$) and $(f_{4,p})$, then any $(PS)$-sequence for $I_{\lambda,p}$ is bounded.
\end{lemma}

In order to obtain positive solutions for problems \eqref{principal2} and \eqref{C1M_1}, we define  $$I_{\lambda,p}^+:X_p^s\to \mathbb{R}$$ $$I_{\lambda,p}^+(u)=\frac{1}{p}\|u\|^p+\frac{\lambda}{q}\int_{\Omega} \vert u^+\vert^q\dd x-\frac{a}{p}\int_{\Omega}\vert u^+\vert^p\dd x-\int_{\Omega} F(u^{+})\dd x.$$
We have that $I_{\lambda,p}^+\in C^1(X_p^s,\mathbb{R})$ and 
\begin{align*}
\langle (I_{\lambda,p}^+)'(u),h\rangle=&\langle A(u),h\rangle+\lambda\int_{\Omega}\!\vert u^+\vert^{q-1}h\dd x-a\int_{\Omega}\!\vert u^+\vert^{p-1}h\dd x\nonumber\\&-\int_{\Omega} f(u^+)h\dd x
\end{align*}
for all $u,h\in X_p^s$. Observe that a critical point for $I_{\lambda,p}^+$ is a weak solution to the problem 
\begin{equation*}
\left\{\begin{array}{l}
(-\Delta)^{s}_{p} u=-\lambda\vert u^+\vert^{q-1}+ a\vert u^{+}\vert^{p-1}+f(u^+) \;\;\;\;\textrm{in}\;\;\Omega,\\
u=0\;\;\;\;\;\;\textrm{in}\;\;\mathbb{R}\setminus \Omega,
\end{array}
\right.
\end{equation*}
where $u^+= \max\{ u,0\}$. It is not difficult to see that a critical point of $I_{\lambda,p}^+$ is a non-negative function of 

We also define $I_{\lambda,p}^-:X_p^s\to \mathbb{R}$ by
$$I_{\lambda,p}^-(u)=\frac{1}{p}\|u\|^p+\frac{\lambda}{q}\int_{\Omega} \vert u^-\vert^q\dd x-\frac{a}{p}\int_{\Omega}\vert u^-\vert^p\dd x-\int_{\Omega} F(u^{-})\dd x.$$
A critical point for $I_{\lambda,p}^-$ is a weak solution to the problem
\begin{equation*}
\left\{\begin{array}{l}
(-\Delta)^{s}_{p} u=-\lambda\vert u^-\vert^{q-1}+ a\vert u^{-}\vert^{p-1}+f(u^-) \;\;\;\;\textrm{in}\;\;\Omega,\\
u=0\;\;\;\;\;\;\textrm{in}\;\;\mathbb{R}\setminus \Omega,
\end{array}
\right.
\end{equation*}
where $u^-= \min\{ u,0\}$ and a non-positive function in $X_p^s$.

By arguments similar to that used in the proof of Lemma \ref{C2_9}, we obtain
\begin{lemma}\label{C2PE_26.1}If $f$ satisfies $(f_{1,p})$, $(f_{2,p})$, $(f_{3,p})$ and $(f_{4,p})$ then any $(PS)$-sequence for $I_{\lambda,p}^{+}$ or $I_{\lambda,p}^{-}$ is bounded.
\end{lemma}
\section{Proof of Theorem \ref{0}}\label{M} 

We start showing a regularization result that will be useful in the proof of our main result.
\begin{lemma}\label{C2P_23.1}
Let $\Omega\subset\mathbb{R}^N$ be a bounded, smooth domain and $f$ a function satisfying $(f_{2,p})$ or $(f'_{2,p})$. Let $(v_\epsilon)_{\epsilon\in(0,1)}\subseteq X_p^s$ be a family of solution to the problem 
\begin{align*}
\left\{\begin{array}{rlll}
(-\Delta)^{s}_{p} u&=&\left(\displaystyle\frac{1}{1-\xi_{\epsilon}}\right)f(u) &\textrm{in}\ \Omega,\\
u&=&0 &\textrm{in}\ \mathbb{R}^N\setminus\Omega,
\end{array}
\right.\end{align*}
where $\xi_\epsilon\leq 0$ and $\|v_\epsilon\|_{X_p^s}\leq 1$, for all  $\epsilon\in(0,1)$. 
Then \[\sup_{\epsilon\in(0,1)}\|v_\epsilon\|_{L^{\infty}(\Omega)}<\infty.\]
\end{lemma}
\begin{proof}
We define, for $0<k\in\mathbb{N}$,
$$T_{k}(s)= \left\{ \begin{array}{rl}
s+k, &\textrm{if}\quad s\leq -k,\\
0, &\textrm{if}\quad -k<s<k,\\
s-k,&\textrm{if}\quad s\geq k
\end{array} \right.$$ 
and
$$\Omega_{k}=\{x\in\Omega\,:\; \vert v_{\epsilon}(x)\vert\geq k\}.$$ 
Observe that $T_{k}(v_{\epsilon})\in X_p^s$ and $\|T_{k}(v_{\epsilon})\|_{X_p^s}^p\leq C^p\|v_{\epsilon}\|_{X_p^s}^p<\infty$ for a constant $C>0$. Taking $T_{k}(v_{\epsilon})$ as a test-function, we obtain
\[\langle A( v_{\epsilon}),T_{k}(v_{\epsilon})\rangle\leq\int_{\Omega} \vert f(v_{\epsilon})\vert \vert T_{k}(v_{\epsilon})\vert \dd x.\]

We claim that
\begin{equation}\label{C2P_23.4}
\langle A(v_{\epsilon}),T_{k}(v_{\epsilon})\rangle_{X_p^s}\leq C\left(\int_{\Omega}\vert T_{k}(v_{\epsilon})\vert^r \dd x\right)^{1/r}\vert\Omega_{k}\vert^{p/r}.
\end{equation}
	
Suppose that $f$ satisfies $(f_{2,p})$. Then, for all $t\in\mathbb{R}$ and $\alpha>0$ we have 
\begin{equation}\label{apriori}\vert f(t)\vert\leq C\exp(\alpha\vert t\vert^{\frac{N}{N-s}})\in L^1(\Omega),\end{equation}
where $C>0$ is a constant. Thus we obtain, for a constant $C_1>0$,
\begin{equation*}
\int_{\Omega} \vert f(v_{\epsilon})\vert \vert T_{k}(v_{\epsilon})\vert \dd x\leq C_1\int_{\Omega} \exp(\alpha \vert v_{\epsilon}\vert^\frac{N}{N-s})\vert T_{k}(v_{\epsilon})\vert \dd x.
\end{equation*} 
If $0<\alpha<\alpha_{s,N}^*$ (see Proposition \ref{C2P_4.4}), we can fix $\theta>1$ so that $0<\theta\alpha<\alpha_{s,N}^*$. Applying the (generalized) Hölder inequality yields \begin{align*}
\int_{\Omega} \vert f(v_{\epsilon})\vert \vert T_{k}(v_{\epsilon})\vert \dd x\leq C_1\left(\int_{\Omega}\exp(\theta\alpha \vert v_{\epsilon}\vert^{\frac{N}{N-s}}) \right)^{1/\theta}\left(\int_{\Omega}\vert T_{k}(v_{\epsilon})\vert^r \dd x\right)^{1/r}\vert\Omega_{k}\vert^{\frac{r-1-\eta}{r}}.
\end{align*}

Since $\|v_{\epsilon}\|_{X_p^s}\leq 1$, it follows from Proposition \ref{C2P_4.4} the existence of a constant $C>0$ such that
\begin{equation*}
\int_{\Omega} \vert f(v_{\epsilon})\vert \vert T_{k}(v_{\epsilon})\vert \dd x\leq C\left(\int_{\Omega}\vert T_{k}(v_{\epsilon})\vert^r \dd x\right)^{1/r}\vert\Omega_{k}\vert^{(r-1-\eta)/r},
\end{equation*}
proving our claim. The proof in the case that $f$ satisfies $(f'_{2,p})$ is analogous. 

Denoting $E=\langle A(v_{\epsilon}),T_{k}(v_{\epsilon})\rangle_{X_p^s}$, we have 
\begin{align*}
E=&\displaystyle{\int_{\mathbb{R}^{2N}}\frac{\vert  v_{\epsilon}(x)-v_{\epsilon}(y)\vert^{p-2}(v_{\epsilon}(x)-v_{\epsilon}(y))(T_{K}(v_{\epsilon})(x)-T_{K}(v_{\epsilon})(y))}{\vert x-y\vert^{N+sp}}}\dd x\dd y\\
=&\int_{\mathbb{R}^N}\left[\int_{v_{\epsilon}(x)\leq -k}\!\!\!\!\!\!\! \frac{\vert  v_{\epsilon}(x)-v_{\epsilon}(y)\vert^{p-2}(v_{\epsilon}(x)-v_{\epsilon}(y))(T_{k}(v_{\epsilon})(x)-T_{k}(v_{\epsilon})(y))}{\vert x-y\vert^{N+sp}}\dd x\right]\dd y\\
&+\int_{\mathbb{R}^N}\left[\int_{\vert v_{\epsilon}(x)\vert<k}\!\!\!\!\!\!\frac{\vert  v_{\epsilon}(x)-v_{\epsilon}(y)\vert^{p-2}(v_{\epsilon}(x)-v_{\epsilon}(y))(T_{k}(v_{\epsilon})(x)-T_{k}(v_{\epsilon})(y))}{\vert x-y\vert^{N+sp}}\dd x\right]\dd y\\
&+\int_{\mathbb{R}^N}\left[\int_{v_{\epsilon}(x)\geq k}\!\!\!\!\!\! \frac{\vert  v_{\epsilon}(x)-v_{\epsilon}(y)\vert^{p-2}(v_{\epsilon}(x)-v_{\epsilon}(y))(T_{k}(v_{\epsilon})(x)-T_{k}(v_{\epsilon})(y))}{\vert x-y\vert^{N+sp}}\dd x\right]\dd y\\
=&\ E_1+E_2+E_3,
\end{align*}
with the integrals $E_i$ defined in the order they appear in the right-hand side.

Let us consider $E_1$. It follows from the definition of $T_{k}$ that
\begin{align*}
E_1=&\int_{v_{\epsilon}(y)\leq -k}\;\left[\int_{v_{\epsilon}(x)\leq -k} \frac{\vert  T_{k}(v_{\epsilon})(x)-T_{k}(v_{\epsilon})(y)\vert^{p}}{\vert x-y\vert^{N+sp}}\dd x\right]\dd y\\
&+\int_{\vert v_{\epsilon}(y)\vert<k}\;\int_{ v_{\epsilon}(x)\leq -k} \frac{\vert  v_{\epsilon}(x)-v_{\epsilon}(y)\vert^{p-2}(v_{\epsilon}(x)-v_{\epsilon}(y))(v_{\epsilon}(x)+k)}{\vert x-y\vert^{N+sp}}\dd x\dd y\\
&+\int_{ v_{\epsilon}(y)\geq k}\;\int_{ v_{\epsilon}(x)\leq -k} \frac{\vert  v_{\epsilon}(x)-v_{\epsilon}(y)\vert^{p-2}(v_{\epsilon}(x)-v_{\epsilon}(y))((v_{\epsilon}(x)-v_{\epsilon}(y) +2k)}{\vert x-y\vert^{N+sp}}\dd x\dd y.
\end{align*}

After some calculation we obtain	
\begin{align*}
E_1\geq & \int_{v_{\epsilon}(y)\leq -k}\;\left[\int_{v_{\epsilon}(x)\leq -k} \frac{\vert  T_{k}(v_{\epsilon})(x)-T_{k}(v_{\epsilon})(y)\vert^{p}}{\vert x-y\vert^{N+sp}}\dd x\right]\dd y\nonumber\\
&+\int_{\vert v_{\epsilon}(y)\vert<k}\;\left[\int_{ v_{\epsilon}(x)\leq -k} \frac{\vert T_{k}(v_{\epsilon})(x)-T_{k}(v_{\epsilon})(x)\vert^{p}}{\vert x-y\vert^{N+sp}}\dd x\right]\dd y\nonumber\\
&+\int_{v_{\epsilon}(y)\geq k}\;\left[\int_{ v_{\epsilon}(x)\leq -k} \frac{\vert  T_{k}(v_{\epsilon})(x)- T_{k}(v_{\epsilon})(y)\vert^{p}}{\vert x-y\vert^{N+sp}}\dd x\right]\dd y\nonumber\\=&\int_{\mathbb{R}^N}\;\left[\int_{ v_{\epsilon}(x)\leq -k} \frac{\vert  T_{k}(v_{\epsilon})(x)- T_{k}(v_{\epsilon})(y)\vert^{p}}{\vert x-y\vert^{N+sp}}\dd x\right]\dd y.
\end{align*}
	
Analogously,
\begin{align*}
E_2\geq\int_{\mathbb{R}^N}\;\left[\int_{\vert v_{\epsilon}(x)\vert<k} \frac{\vert  T_{k}(v_{\epsilon})(x)- T_{k}(v_{\epsilon})(y)\vert^{p}}{\vert x-y\vert^{N+sp}}\dd x\right]\dd y.
\end{align*} 
and
\begin{align*}
E_3\geq\int_{\mathbb{R}^N}\;\left[\int_{ v_{\epsilon}(x)\geq k} \frac{\vert  T_{k}(v_{\epsilon})(x)- T_{k}(v_{\epsilon})(y)\vert^{p}}{\vert x-y\vert^{N+sp}}\dd x\right]\dd y,
\end{align*}
so that	
\begin{align*}
E_1+E_2+E_3\geq \int\limits_{\mathbb{R}^{2N}}\frac{\vert  T_{k}(v_{\epsilon})(x)- T_{k}(v_{\epsilon})(y)\vert^{p}}{\vert x-y\vert^{N+sp}}\dd x\dd y,
\end{align*}
thus yielding	
\begin{equation}\label{CPL2_23.4.16.0}
\langle A(v_{\epsilon}),T_{k}(v_{\epsilon})\rangle\geq\|T_{k}(v_{\epsilon})\|_{X_p^s}^{p}.
\end{equation}
	
The continuous immersion $X_p^s\hookrightarrow L^{r}(\Omega)$ gives us (for a constant $C_1>0$)
\begin{equation}\label{CPL2_23.5000}
C_1 \left(\int_{\Omega}\vert T_{k}(v_{\epsilon})\vert^{r} \dd x\right)^{p/r}\leq\langle A(v_{\epsilon}),T_{k}(v_{\epsilon})\rangle.
	\end{equation}
	
Thus, it follows from \eqref{C2P_23.4} and \eqref{CPL2_23.5000} the existence of  $C>0$ such that 
\begin{equation*}
\int_{\Omega}\vert T_{k}(v_{\epsilon})\vert^r \dd x\leq C\vert\Omega_{k}\vert^{p/(p-1)}.
\end{equation*}
	
Since, for all $s\in\mathbb{R}$, we have $\vert T_k(s)\vert=(\vert s\vert-k)(1-\chi_{[-k,k]}(s))$, we conclude that, if $0<k<h\in\mathbb{N}$, then $\Omega_{h}\subset\Omega_k$. Thus,
\begin{align*}
\int_{\Omega}\vert T_{k}(v_{\epsilon})\vert^r \dd x =&\int_{\Omega_k}(\vert v_{\epsilon}\vert-k)^r
\geq \int_{\Omega_h}(\vert v_{\epsilon}\vert-k)^r
\geq (h-k)^r\vert\Omega_h\vert.
\end{align*}
	
Defining, for $0<k\in\mathbb{N}$, $$\phi(k)=\vert\Omega_k\vert,$$ 
we obtain
\begin{equation*}
\phi(h)\leq C(h-k)^{-r}\phi(k)^{p/(p-1)},\quad 0<k<h\in\mathbb{N}.
\end{equation*}
	
By induction we obtain 
$$\phi(k_n)\leq\frac{\phi(0)}{2^{nr(p-1)}},\quad\textrm{ for all }n\in\mathbb{N},$$ 
from what follows  
$\phi(d)=0$ and, consequently,$$\vert v_\epsilon(x)\vert\leq d\quad\textrm{ a.e. }\,\textrm{ in }\Omega,\,\textrm{ for all } \epsilon\in(0,1).$$
Thus, $$\|v_\epsilon\|_{L^{\infty}(\Omega)}\leq d\;\textrm{ for all }\,\epsilon\in(0,1)$$ and we are done. 
$\hfill\Box$\end{proof}

We specify the spaces $C_{\delta}^0(\overline{\Omega})$ and $C_{\delta}^{0,\alpha}(\overline{\Omega})$. For this, we define $\delta:\overline{\Omega}\to\mathbb{R}^+$ by
$\delta(x)=\textup{dist}(x,\mathbb{R}^N\setminus\Omega)$. Then, if $0<\alpha<1$,
\begin{align*}
C_{s}^{0}(\overline{\Omega})&=\left\{u\in C^{0}(\overline{\Omega})\;:\;\frac{u}{\delta^{s}}\;\textrm{ has a continuous extension to } \overline{\Omega}\right\}\\
C_{s}^{0,\alpha}(\overline{\Omega})&=\left\{u\in C^0(\overline{\Omega})\;:\;\frac{u}{\delta^{s}}\;\textrm{ has a $\alpha$-Hölder extension to } \overline{\Omega}\right\}
\end{align*}
with the respective norms \begin{align*}\|u\|_{0,s}&=\left\|\displaystyle\frac{u}{\delta^{s}}\right\|_{L^{\infty}(\Omega)}\\
\intertext{and}
\|u\|_{\alpha,s}&=\|u\|_{0,\delta}+\displaystyle{\sup_{x,y\in\overline{\Omega},\,x\neq y.}}\frac{\left\vert u(x)/\delta(x)^{s}- u(y)/\delta(y)^{s}\right\vert}{\vert x-y\vert^{\alpha}}.
\end{align*}

\textit{Proof of Theorem \ref{0}.} For $0<\epsilon<1$, let us denote $B_{\epsilon}=\{z\in X_p^s\;:\;\|z\|_{X_p^s}\leq\epsilon\}$. 
By contradiction, suppose that for each  $\epsilon>0$, there exists $u_\epsilon\in B_{\epsilon}$ such that 
\begin{equation}\label{C2P_24.1}
\Phi(u_\epsilon)<\Phi(0).
\end{equation} 

It is not difficult to verify that  $\Phi:B_{\epsilon}\to\mathbb{R}$ is weakly lower semicontinuous. Therefore, 
there exists $v_\epsilon\in B_{\epsilon}$ such that $\displaystyle{\inf_{u\in B_{\epsilon}}}\Phi(u)=\Phi(v_\epsilon)$. It follows from \eqref{C2P_24.1} that $$\Phi(v_\epsilon)=\displaystyle{\inf_{u\in B_{\epsilon}}}\Phi(u)\leq\Phi(u_\epsilon)<\Phi(0).$$

We will show that 
\begin{equation*}
v_{\epsilon}\rightarrow 0 \textrm{ in } C_s^0(\overline{\Omega}) \textrm{ as }~~ \epsilon\rightarrow 0,
\end{equation*}
since this implies that, for $r_1>0$, 
the existence of $ z \in C_s^0(\overline{\Omega})$, such that $\|z\|_{0,s}<r_1$ and 
$\Phi(z)<\Phi(0)$, contradicting \eqref{PH1}.

Since $v_\epsilon$ is a critical point of $\Phi$ in $X^s_p$, by Lagrange multipliers we obtain the existence of $\xi_{\epsilon}\leq 0$ such that
\begin{equation*}
\langle \Phi'(v_{\epsilon}), \phi \rangle = \xi_{\epsilon}\langle v_{\epsilon}, \phi \rangle,\quad \forall\ \phi \in X_p^s.
\end{equation*}

Thus, $v_{\epsilon}$ satisfies 
\begin{align*}
(-\Delta)_{p}^s v_{\epsilon} &= \left(\displaystyle\frac{1}{1-\xi_{\epsilon}}\right)g(v_{\epsilon}) =: g^{\epsilon}(v_{\epsilon})  &\textrm{ in } &\Omega,\\
v_{\epsilon}&=0  &\textrm{ in } &\mathbb{R}\backslash \Omega,
\end{align*}

If $ \|v_{\epsilon}\|_{X_p^s} \leq\epsilon<1$, Proposition \ref{C2P_23.1} show the existence of a constant $C_1>0$, not depending on $\epsilon$, such that
\begin{equation}\label{C2P_24.11.0}
\|v_{\epsilon}\|_{L^{\infty}(\Omega)} \leq C_1.
\end{equation}

Since $\xi_{\epsilon}\leq 0$,  \eqref{apriori} and \eqref{C2P_24.11.0} show that 
\begin{equation*}
\|g^{\epsilon}(v_{\epsilon})\|_{L^{\infty}(0,1)}\leq C_2
\end{equation*}
for some constant $C_2>0$.

Theorem \ref{a.3} then yields $\|v_{\epsilon}\|_{C^{0,\beta}(\overline{\Omega})} \leq C_3$, for $0<\beta\leq s$ and a constant $C_3$ not depending on $\epsilon$.

It follows from Arzelà-Ascoli theorem the existence of a \textit{sequence} $(v_\epsilon)$ such that $v_{\epsilon}\rightarrow 0$ uniformly as $\epsilon\rightarrow 0$. Passing to a subsequence, we can suppose that $v_\epsilon\to 0$ a. e. in $\Omega$ and, therefore, $v_{\epsilon}\to 0$, uniformly in $\overline{\Omega}$.
But now follows from Proposition \ref{a.4} that 
\begin{equation*}
\|v_{\epsilon}\|_{0,\delta}=\biggl\| \frac{v_{\epsilon}}{\delta^{s}} \biggl\|_{L^{\infty}(\Omega)} \leq  C \sup_{x\in (0,1)} |g^{\epsilon}(v_{\epsilon}(x))|
\end{equation*}
for a constant $C>0$. We are done.
$\hfill\Box$

\begin{remark}
Observe that, if $0$ a strict local minimum in $C_{\delta}^{0}(\overline{\Omega})$, then $0$ is also a strict local minimum in $X_p^s$.
\end{remark}

In the case $p=2$, $s=1/2$ and $N=1$ we have the following result, analogous to Lemma \ref{C2P_23.1}: 
\begin{lemma}
Let $(v_\epsilon)_{\epsilon\in(0,1)}\subseteq X$ be a family of solutions to
\begin{align*}
\left\{
\begin{array}{rlll}
(-\Delta)^{1/2} u&=&\left(\displaystyle\frac{1}{1-\xi_{\epsilon}}\right)f(u) &\textrm{in}\ (0,1),\\
u&=&0 &\textrm{in}\;\;\mathbb{R}\setminus(0,1),
\end{array}
\right.
\end{align*}
where $\xi_\epsilon\leq 0$, with $\|v_\epsilon\|_{X}\leq 1$ for all $\epsilon\in(0,1)$ and with $f$ satisfying one of the following conditions:
\begin{enumerate}
\item[$(f_{2,2})$] For all $\alpha>0$, $$\lim_{\vert t\vert\to \infty}\frac{ \vert f(t)\vert}{\exp(\alpha t^2)}=0;$$ 
\item[$(f'_{2,2})$] there exists $\alpha_0>0$ such that 
$$\lim_{\vert t\vert\to \infty}\frac{ \vert f(t)\vert}{\exp(\alpha t^2)}= \left\{ \begin{array}{rc}
\infty, &\quad\;\textrm{if}\quad 0<\alpha<\alpha_0 \\
0, &\textrm{if}\quad \alpha>\alpha_0.
\end{array} \right.$$
\end{enumerate}
Then $$\sup_{\epsilon\in(0,1)}\|v_\epsilon\|_{L^{\infty}}<\infty.$$
\end{lemma}
The proof follows the same reasoning of Lemma \ref{C2P_23.1}. Observe that the proof of the estimate \eqref{CPL2_23.4.16.0} is simpler:
\begin{align*}
\langle v_{\epsilon},T_{k}(v_{\epsilon})\rangle_{X}&=\langle v_{\epsilon}\pm k,T_{k}(v_{\epsilon})\rangle_{X}\mp\langle k,T_k(v_\epsilon)\rangle_{X}\\
&=\langle T_{k}(v_{\epsilon}),T_{k}(v_{\epsilon})\rangle_{X}=\|T_{k}(v_{\epsilon})\|_{X}^2.
\end{align*}

The proof of Theorem \ref{C2_24} is analogous to that of Theorem \ref{0}.
\section{Proof of Theorem \ref{3}}\label{subcritico}

In this section we deal with existence and multiplicity of solutions to the problem  \eqref{principal2} when $f$ has subcritical growth. 

The proof of Theorem \ref{3} will be given in 3 subsections. In the first subsection, we will obtain a positive solution by applying the Mountain Pass Theorem. Analogously, in the second subsection we will obtain a negative solution. In the last subsection, a third solution will be obtained by the Linking Theorem and we conclude the proof of Theorem \ref{3}. 
\subsection{Positive solution for the  functional $I_{\lambda,p}$}
\begin{lemma}\label{C3PE_6.1}
Suppose that $f$ satisfies $(f_{1,p})$, $(f_{2,p})$, $(f_{3,p})$ and $(f_{4,p})$. Then, for any $\lambda>0$, the functional $I_{\lambda,p}^{+}$ satisfies the $(PS)$ condition at any level. 

The same result is valid for the functional $I_{\lambda,p}$. 
\end{lemma}
\begin{proof}
Let $(u_n)\subset X_p^s$ be a $(PS)$-sequence for $I^+_{\lambda,p}$. By Lemma \ref{C2PE_26.1}, there exists $u_0\in X_p^s$ such that 
$u_n\rightharpoonup u_0\;\;\textrm{ in }\  X_p^s$.

We can also suppose that
\begin{equation*}
u_n\to u_0\;\;\textrm{ in }\ L^r(\Omega)\;\;\textrm{ for } r\geq 1\ \textrm{ and }\ u_n(x)\to u_0(x)\ \textrm{ a.e. in }\ \Omega.
\end{equation*} 

For $1<q<2\leq p$, by applying Hölder's inequality we obtain
\begin{align*}
\int_{\Omega}\vert u_n^+\vert^{q-2}u_n^+(u_n-u_0) \to 0\quad\textrm{and}\quad \int_{\Omega}\vert u_n^+\vert^{p-2}u_n^+(u_n-u_0) \to 0.
\end{align*}
Observe that 
\begin{equation*}
u_n-u_0\rightharpoonup 0,\ \textrm{ implies }\ \langle (I_{\lambda,p}^+)'(u_n),u_n-u_0\rangle\to 0.
\end{equation*}
It follows that
\begin{align*}
\langle A(u_n),u_n-u_0\rangle=&\langle (I_{\lambda,p}^{+})'(u_n),u_n-u_0\rangle-\lambda\int_{\Omega}\vert u_n^{+}\vert^{q-2}u_n^{+}(u_n-u_0)\\&+a\int_{\Omega}\vert u_n^{+}\vert^{p-2}u_n^{+}(u_n-u_0)+\displaystyle\int_{\Omega} f(u_n^+)(u_n-u_0)\\
=&\displaystyle\int_{\Omega}f(u_n^+)(u_n-u_0)+o(1).
\end{align*}

Taking $0<\alpha<\displaystyle\frac{\alpha_{s,N}^*}{rM^\frac{N}{N-s}}$, the inequality 
$$\vert f(t)\vert\leq C\exp(\alpha\vert t\vert^{\frac{N}{N-s}}),\ \textrm{ for all }\ t\in\mathbb{R}.$$
implies the existence of $C>0$ such that
\begin{align*}
\langle A(u_n),u_n-u_0\rangle&\leq \displaystyle\int_{\Omega}\vert f(u_n^+)\vert\vert u_n-u_0\vert+o(1)\\
&\leq\displaystyle C\int_{\Omega}\exp({\alpha \vert u_n\vert^\frac{N}{N-s}})\vert u_n-u_0\vert+o(1)
\end{align*}
Thus, since $0<r\alpha M^\frac{N}{N-s}<\alpha_{s,N}^*$, it follows from Hölder's inequality
$$\langle A(u_n),u_n-u_0\rangle\leq C C_{1} \left(\int_{\Omega}\vert u_n-u_0\vert^{r/(r-1)}\right)^{(r-1)/r}+o(1)$$
for a positive constant $C_1$. Thus
$$\langle A(u_n),u_n-u_0\rangle\to 0 $$
and we conclude $u_n\to u_0$ in $X_p^s$ as a consequence of Lemma \ref{C2P_4.6}.

In the case of the functional $I_{\lambda,p}$ the proof is analogous.
$\hfill\Box$\end{proof}

The next results will be useful when proving the geometric conditions of the Mountain Pass Theorem. We define
\begin{align*}
J_{\lambda,p}(u)
&:= I_{\lambda,p}(u)-\frac{1}{p}\|u\|_{X_p^s}^p= \frac{\lambda}{q}\int_{\Omega}\vert u\vert^q\dd x-\frac{a}{p}\int_{\Omega}\vert u\vert^p\dd x-\int_{\Omega}F(u)\dd x
\intertext{and}
J_{\lambda,p}^+(u)
&:= I_{\lambda,p}^+(u)-\frac{1}{p}\|u\|_{X_p^s}^p= \frac{\lambda}{q}\int_{\Omega}\vert u^+\vert^q\dd x-\frac{a}{p}\int_{\Omega}\vert u^+\vert^p\dd x-\int_{\Omega}F(u^+)\dd x.
\end{align*}
\begin{lemma}[Subcritical and critical cases]\label{C3PE_4}
Suppose that $a>0$ and that $f$ satisfies $(f_{3,p})$. Then, the trivial solution $u=0$ is a strict local minimum of  $J_{\lambda,p}^{+}$ for all $\lambda>0$.
\end{lemma}
\begin{proof}
According to Theorem \ref{0}, it suffices to show that $u=0$ is a strict local minimum for $J_{\lambda,p}^{+}$ in  $C^0_{\delta}(\overline{\Omega})$. Condition $(f_{3,p})$ implies, for some $\omega>0$,
\begin{equation*}
\displaystyle\lim_{\vert t\vert\to 0}\frac{ F(t)}{ \vert t\vert^p}=0\quad\Rightarrow\quad \vert F(t)\vert< \vert t\vert^p,\;\textrm{ for all } 0<\vert t\vert\leq\omega.
\end{equation*}

Consider $u\in (C^0_\delta(\overline{\Omega})\cap X_p^s)\setminus\{0\}$. Taking  com $\|u\|_{0,\delta}$ small enough, we have  $0<\vert u^+\vert<\omega$, since $\vert u^+\vert\leq M\|u\|_{0,\delta}$ for some $M>0$. Thus,
\begin{align}\label{C3PE_5}
J_{\lambda,p}^+(u)
&=\frac{\lambda}{q}\int_{\Omega}\vert u^+\vert^q\dd x-\left(\frac{a}{p}+1\right)\int_{\Omega}\vert u^+\vert^p\dd x
\end{align}

For $1<q<p$, we have $\vert u^{+}\vert^{p-q}\leq (k_1)^{p-q}\| u\|^{p-q}_{0,\delta}$ for some constant $k_1>0$. Thus, $$\int_{\Omega}\vert u^{+}\vert^p \dd x\leq(k_1)^{p-q}\|u\|^{p-q}_{0,\delta}\int_{\Omega}\vert u^+\vert^q \dd x$$
and we obtain 
\begin{equation}\label{C3PE_700}
-\left(\frac{a}{p}+1\right)\displaystyle\int_{\Omega}\vert u^+\vert^p\dd x \geq-\left(\frac{a}{p}+1\right)(k_1)^{p-q}\| u\|_{0,\delta}^{p-q}\displaystyle\int_{\Omega}\vert u^+\vert^q\dd x,
\end{equation}
since $\dfrac{a}{p}+1>0$.  

So, \eqref{C3PE_5} and \eqref{C3PE_700} yield
\begin{align*}
J_{\lambda,p}^+(u)\geq
&\left[ \frac{\lambda}{q}-\left(\frac{a}{p}+1\right)(k_1)^{p-q}\| u\|_{0,\delta}^{p-q}\right] \int_{\Omega}\vert u^+\vert^q \dd x.
\end{align*}
By defining  
$$R= \dfrac{(\lambda/q)^{1/(p-q)} }{k_1\left(\dfrac{a}{p}+1\right)^{1/(p-q)}},$$ 
if $\;\| u\|_{0,\delta}<R$, then 
$$J_{\lambda,p}^+(u)> 0=J_{\lambda,p}^+(0),\;\forall\;0<\| u\|_{0,\delta}<R,$$ completing the proof.
$\hfill\Box$\end{proof}
\begin{remark}
The same result holds for $J_{\lambda,p}$.
\end{remark}

\begin{lemma}[Subcritical and critical cases]\label{C2PE_34}
Suppose that $a>0$ and $f$ satisfies $(f_{4,p})$. Then, for a fixed $\Lambda>0$, there exists $t_0=t_0(\Lambda)$ such that $$I_{\lambda,p}^+(t\varphi_1)<0,$$ for all $t\geq t_0$ and $0<\lambda<\Lambda$.
\end{lemma}
\begin{proof}It follows from $(f_{4,p})$ that, fixed $M>0$, there exists $C_M>0$ such that \begin{equation}\label{C5P_10.0}F (t)\geq M |t|^p - C_M.\end{equation} Thus, if $M>\dfrac{\lambda_1}{p}$, denoting by $\varphi_1$ the positive eigenfunction associated with the eigenvalue $\lambda_1$, with $\|\varphi_1\|_{L^p(\Omega)}=1$, we have
\begin{align*}
I_{\lambda,p}^+(t\varphi_1)
&\leq\frac{\vert t\vert^p}{p}\|\varphi_1\|_X^p+\frac{\vert t\vert^q\lambda}{q}\int_{\Omega}\vert\varphi_1\vert^q \dd x-\int_{\Omega}F(t\varphi_1)\dd x\nonumber\\
&\leq \frac{\vert t\vert^p}{p}\|\varphi_1\|_X^p+\frac{\vert t\vert^q\lambda}{q}\int_{\Omega}\vert \varphi_1\vert^q \dd x-\int_{\Omega}\left(M \vert t\vert^p\vert \varphi_1\vert^p+C_M\right)\dd x\nonumber\\
&\leq\frac{\lambda_1 \vert t\vert^p}{p}\int_{\Omega} \vert\varphi_1\vert^p \dd x +\frac{\vert t\vert^q\lambda}{q}\int_{\Omega}\vert\varphi_1\vert^q \dd x-M\vert t\vert^p\int_{\Omega}\vert\varphi_1\vert^p\dd x+C_M\vert\Omega\vert\nonumber \\
&= t^p\left[\frac{\lambda}{q}\frac{1}{t^{p-q}}\int_\Omega\vert\varphi_1\vert^q \dd x+\frac{1}{t^p} C_M\vert\Omega\vert-\left(M-\frac{\lambda_1}{p}\right)\right].
\end{align*}

For a fixed $\Lambda>0$ we now choose $t_0=t_0(\Lambda)>0$ such that
$$\frac{\Lambda}{q}\frac{1}{t_0^{p-q}}\int_\Omega\varphi_1^q \dd x+\frac{C_M}{t_0^p}\vert\Omega\vert - \left(M-\frac{\lambda_1}{p}\right)<0.$$

So, for $t\geq t_0$ and $\lambda<\Lambda$ we have 
\begin{align*}\frac{\lambda}{q}\frac{1}{t^{p-q}}\int_\Omega\varphi_1^q \dd x+\frac{C_M}{t^p}\vert\Omega\vert - \left(M-\frac{\lambda_1}{p}\right)  &\leq\frac{\Lambda}{q}\frac{1}{t_0^{p-q}}\int_\Omega\varphi_1^q \dd x+\frac{C_M}{t_0^p}\vert\Omega\vert - \left(M-\frac{\lambda_1}{p}\right)\\
&<0.
\end{align*}
The result follows.
$\hfill\Box$\end{proof}

\begin{proposition}
Suppose that, for $a>0$, $f$ satisfies $(f_{1,p})$, $(f_{2,p})$, $(f_{3,p})$  $(f_{4,p})$. Then, the subcritical problem  \eqref{principal2} has at least one positive solution, for all $0<\lambda<\Lambda$, where $\Lambda>0$ is arbitrary.
\end{proposition}
\begin{proof}
It follows immediately from the Mountain Pass Theorem, as consequence of Lemmas \ref{C3PE_6.1}, \ref{C3PE_4} and \ref{C2PE_34}. 
$\hfill\Box$\end{proof}
\subsection{Negative solution for the functional $I_{\lambda,p}$}
The Palais-Smale condition is obtained by following the reasoning given in the proof of Lemma \ref{C3PE_6.1}.

For $u\in X_p^s$, by defining
\begin{align*}
J_{\lambda,p}^-(u)
&:= I_{\lambda,p}^-(u)-\frac{1}{p}\|u\|_{X_p^s}^p= \frac{\lambda}{q}\int_{\Omega}\vert u^-\vert^q\dd x-\frac{a}{p}\int_{\Omega}\vert u^-\vert^p\dd x-\int_{\Omega}F(u^-)\dd x,
\end{align*}
we obtain the result analogous to Lemma \ref{C3PE_4}. 

Mimicking the proof of Lemma \ref{C2PE_34}, we obtain the second condition of the Mountain Pass Theorem.
Thus, the negative solution follows, as before, from the Mountain Pass Theorem. 
\subsection{A third solution}
In order to obtain the geometric conditions of the Linking Theorem, we define
\begin{equation*}
\lambda^{*}=\inf\left\{\|u\|_{X_p^s}^{p}\; : \; u\in W, \ \|u\|_{L^{p}(\Omega)}^{p}=1\right\},
\end{equation*}
where $$W=\left\{ u\in X_p^s\; :\; \langle A(\varphi_1),u\rangle=0\right\},$$
with $\varphi_1$ the first autofunction, positive and normalized, of $(-\Delta)_p^s$. The proof of the next result is simple.
\begin{proposition}
$X_p^s=W\oplus\textup{span}\{\varphi_1\}$.
\end{proposition}

Following ideas of Alves, Carrião and Miyagaki \cite{alves} and Capozzi, Fortunato and  Palmieri \cite{Capozzi}, we obtain the next result.
\begin{proposition}
$\lambda_1<\lambda^{*}$.
\end{proposition}
\begin{proof}Of course
$$\lambda_1=\inf\bigg\{\|u\|_{X_p^s}^p\; :\; \|u\|_{L^p(\Omega)}^p=1\bigg\}\leq\inf\bigg\{\|u\|_{X_p^s}^p\; :\; u\in W\;\textrm{ e }\;\|u\|_{L^p(\Omega)}^p=1\bigg\}=\lambda^*.$$ 

Suppose that $\lambda_1=\lambda^{*}$. It follows the existence of a sequence $(u_n)\subset W$ such that $$\|u_n\|_{L^p(\Omega)}^p=1\quad \textrm{and}\quad \displaystyle{\lim_{n\to\infty}}\|u_n\|_{X_p^s}^p=\lambda^{*}=\lambda_1.$$ 

Since $(u_n)$ is bounded in $X_p^s$, passing to a subsequence if necessary, there exists $u\in X_p^s$ such that
$$\begin{cases}
u_n\rightharpoonup u\;\textrm{ in }\; X_p^s,\\
u_n\to u \textrm{ in }\; L^{q}(\Omega),\  1\leq q\leq p\\
u_n(x)\to u(x)\textrm{ a.e. in } \Omega.
\end{cases}
$$
Since 
\begin{align*}
\lambda_1\leq\|u\|_{X_p^s}^p &\leq\displaystyle{\liminf_{k\to\infty}}\|u_k\|^p=\displaystyle{\liminf_{k\to\infty}}\lambda_{k}=\lambda_1.
\end{align*}
we conclude that $u=t\varphi_1$ for some $t\neq 0$.

But $\langle A(\varphi_1),u_n\rangle\to \langle A(\varphi_1),u\rangle$. Since $(u_n)\subset W$, we have $\langle A(\varphi_1),u_n\rangle=0$, thus implying  $\langle A(\varphi_1),u\rangle=0$. It follows $t\|\varphi_1\|_{X_p^s}^p=0$, thus implying $t=0$, and we have reached a contradiction.
$\hfill\Box$\end{proof}

\begin{lemma}\label{C2PE_51} 
If $a<\lambda^{*}$, then there exist $\beta, \rho>0$ such that $I_{\lambda,p}(u)\geq\beta$ for all $u\in W$ such that $\|u\|_{X_p^s}=\rho$.
\end{lemma}
\begin{proof}
Take $\theta>p$ and $0<\alpha<\alpha_{s,N}^*$ (see Proposition \ref{C2P_4.4}).  It follows from $(f_{2,p})$ and $(f_{3,p})$ the existence of $0\leq\mu<\lambda^{*}-a$ and $C>0$ such that $$F(t)\leq \frac{\mu}{p}t^P+C \exp(\alpha \vert t\vert^\frac{N}{N-s})\vert t\vert^\theta,\;\textrm{ for all } t\in\mathbb{R}.$$
Thus, if $u\in W$ and $\|u\|_{X_p^s}\leq 1$, then
\begin{align*}
I_{\lambda,p}(u)
&\geq \frac{\|u\|_{X_p^s}^p}{p} -\frac{a}{p}\int_{\Omega}\vert u\vert^p\dd x-\int_{\Omega}\left[\frac{\mu\vert u\vert^p}{p}+C \exp(\alpha \vert u\vert^\frac{N}{N-s})\vert u\vert^\theta \right]\dd x\\
&=\frac{\|u\|_{X_p^s}^p}{p}-\left(\frac{a}{2}+\frac{\mu}{2}\right)\int_{\Omega}\vert u\vert^p\dd x-C\int_{\Omega} \exp(\alpha \vert u\vert^\frac{N}{N-s})\vert u\vert^\theta \dd x.
\end{align*}
The definition of $\lambda^{*}$ yields
\begin{equation*} 
I_{\lambda,p}(u)\geq \frac{\|u\|_{X_p^s}^p}{p}-(a+\mu)\frac{\|u\|_{X_p^s}^p}{p\lambda^{*}}-C\int_{\Omega}\exp(\alpha u^2)\vert u\vert^\theta \dd x.
\end{equation*}

Take $r>1$ so that $0<r\alpha<\alpha_{s,N}^*$. Hölder's inequality then implies 
\begin{align*}
I_{\lambda,p}(u)&\geq\left(\frac{1}{p}-\frac{a+\mu}{p\lambda^{*}}\right)\|u\|_{X_p^s}^p-C\left(\int_{\Omega} \exp(\alpha r\vert u\vert^\frac{N}{N-s})\dd x\right)^{1/r}\left(\int_{\Omega} \vert u\vert^{\theta r'} \dd x\right)^{1/r'}\\
&=\frac{1}{p}\left(1-\frac{a+\mu}{\lambda^{*}}\right)\|u\|_{X_p^s}^p-C\left(\int_{\Omega} \exp(\alpha r \vert u\vert^\frac{N}{N-s}) \dd x\right)^{1/r}\|u\|_{L^{r'\theta}(\Omega)}^\theta.
\end{align*}

Since $\|u\|_{X_p^s}\leq 1$, the continuous immersion $X_p^s\hookrightarrow L^{r'q}(\Omega)$ and Proposition  \ref{C2P_4.4} guarantee the existence of $C_1>0$ such that $$I_{\lambda,p}(u)\geq\frac{1}{p}\left(1-\frac{a+\mu}{\lambda^{*}}\right)\|u\|_{X_p^s}^2-C_1\|u \|_{X_p^s}^\theta.$$

Observe that $\theta>p$ and $1-\displaystyle\frac{a+\mu}{\lambda^{*}} >0$. Thus, for $\rho>0$ small enough and  $\|u\|_{X_p^s}=\rho$, we have $$I_{\lambda,p}(u)\geq \rho^p\left\{\frac{1}{p}\left(1-\frac{a+\mu}{\lambda^{*}}\right)-C_1\rho^{\theta-p}\right\}>0.$$ 

The proof is complete by defining  $\beta=\rho^p\left\{\displaystyle\frac{1}{p}\left(1-\frac{a+\mu}{\lambda^{*}}\right)-C_1\rho^{\theta-p}\right\}$. $\hfill\Box$\end{proof}

\begin{lemma}\label{C2PE_52}  
Suppose that $f$ satisfies $(f_{4,p})$. If $Y\subset X_p^s$ is a subspace with $\dim Y<\infty$, then 
$$\displaystyle\lim_{ u\in Y, \Vert u\Vert_{X_p^s}\to \infty}I_{\lambda,p}(u)=-\infty.$$
\end{lemma}
\begin{proof}
Since all norms in $Y$ are equivalent, there exist $C_1>0$ and $C_2>0$ such that, for all $y\in Y$ holds
\begin{equation}\label{C2_530}
 \|u\|_{X_p^s}^p\leq C_1\|u\|_{L^p(\Omega)}^p\quad\textrm{and}\quad\Vert u\Vert_{L^q(\Omega)}^q\leq C_2\|u\|_{X_p^s}^q.
\end{equation}

Now, by applying \eqref{C5P_10.0} for $M>C_1\left(\displaystyle\frac{1}{p}+\frac{\lambda C_2}{q}\right)$ and \eqref{C2_530}, we obtain 
\begin{align*}
I_{\lambda,p}(u)&\leq\frac{\|u\|_{X_p^s}^p}{p}+\frac{\lambda}{q}\int_{\Omega}\vert u\vert^q\dd x-\int_{\Omega}(M\vert u\vert^p-C_M)\dd x\\
&\leq\frac{\|u\|_{X_p^s}^p}{p}+\frac{\lambda}{q}\|u\|_{L^q(\Omega)}^q-M\Vert u\Vert_{L^p(\Omega)}^p+C_M\vert\Omega\vert\\
&\leq\|u\|_{X_p^s}^p\left(\frac{1}{p}+\frac{\lambda C_2}{q}-\frac{M}{C_1}\right)+C_M\vert\Omega\vert.   
\end{align*}

The choice of $M$ implies the result. 
$\hfill\Box$\end{proof}

\begin{lemma}\label{C3PE_9}
If $f$ satisfies $(f_{4,p})$ and if $\lambda_1\leq a<\lambda^{*}$, then there exists $\eta=\eta(\lambda)>0$ such that $$I_{\lambda,p}(u)\leq \eta(\lambda)\quad\textrm{and}\quad\displaystyle{\lim_{\lambda\to 0}}\;\eta(\lambda)=0,\quad\textrm{ for all }\ u\in\textup{span}\{\varphi_1\}.$$
\end{lemma}
\begin{proof}Since $u\in\textrm{span}\{\varphi_1\}$, we have 
\begin{align*}
I_{\lambda,p}(u)
&\leq\left( \frac{\lambda_1-a}{p}\right)\int_{\Omega} u^p\dd x +\frac{\lambda}{q}\Vert u\Vert_{L^q(\Omega)}^q-\int_\Omega F(u)\dd x.  
\end{align*}

But $\lambda_{1}\leq a<\lambda^{*}$, the continuous immersion and \eqref{C5P_10.0} imply that
\begin{align*}
I_{\lambda,p}(u)&\leq\frac{\lambda}{q}\Vert u\Vert_{L^q(\Omega)}^q-\int_{\Omega} F(u)\dd x\\
&\leq\frac{\lambda}{q}K_{q}\Vert u\Vert_{X_p^s}^q-\int_{\Omega} F(u)\dd x\\
&\leq \frac{\lambda K_q}{q}\|u\|_{X_p^s}^q-M\|u\|_{L^p(\Omega)}^p+C_M\\
&\leq\frac{\lambda K_q}{q}\|u\|_{X_p^s}^q-MK_2\|u\|_{X_p^s}^p+C_M.   
\end{align*}

Since $1<q<p$, we conclude that 
\begin{equation*}
\displaystyle\lim_{ u\in\textrm{span}\{\varphi_1\} , \Vert u\Vert_{X_p^s}\to \infty}I_{\lambda,p}(u)=-\infty.
\end{equation*}

Thus, there exists $R>0$ such that 
$I_{\lambda,p}(u)<0$ for all $u\in \textrm{span}\{\varphi_1\}$ satisfying $ \|u\|_{X_p^s}>R$.

If $u\in \textrm{span}\{\varphi_1\}$ and $\|u\|_{X_p^s}\leq R$, we have
\begin{align*}
0\leq I_{\lambda,p}(u)
\leq \frac{\lambda}{q}K_{q}\Vert u\Vert_{X_p^s}^q-\int_{\Omega} F(u)\dd x
\leq \frac{\lambda}{q}K_{q}R^q-\int_{\Omega} F(u)\dd x\leq \frac{\lambda}{q}K_{q}R^q.
\end{align*}

The result follows by defining $\eta(\lambda)=\displaystyle\frac{\lambda}{q}K_{q}R^q$. $\hfill\Box$\end{proof} 

\begin{proposition}\label{proplinking}Suppose that $f$ satisfies  $(f_{1,p})-(f_{4,p})$. If $\lambda_{1}\leq a<\lambda^*$, then problem \eqref{principal2} has at least a third solution, for all $\lambda>0$ small enough. 
\end{proposition}
\begin{proof}
We already know that  $X_p^s=W\oplus\textup{span}\{\varphi_1\}$ and that the functional $I_{\lambda,p}$ satisfies the Palais-Smale condition at all levels, for any $\lambda>0$. 
Therefore, the Linking theorem guarantees that $I_{\lambda,p}$ has a critical value $C\geq\beta$ given by $$C=\displaystyle\inf_{\gamma\in\Gamma}\displaystyle\max_{u\in Q} {I_{\lambda,p}(\gamma(u))}$$ where $\Gamma=\{\gamma\in C(\bar{Q},E)\;\;;\;\;\gamma=I_d\ \textrm{ in }\ \partial Q\}$.

Taking into account Lemma \ref{C2PE_51}, to conclude our result from the Linking Theorem, it suffices to show the existence of $e\in(\partial B_1)\cap W$, constants $R>\rho$ and $\alpha>0$ such that $I_{\lambda,p}\big|_{\partial Q}<\alpha<\beta$, where $Q=(B_R\cap \textup{span}\{\varphi_1\})\oplus(0,Re)$.
 
So, take 
$\varphi\in W$ with $\|\varphi\|_{X_p^s}=1$. Lemma \ref{C2PE_52} guarantees the existence of $\bar{R}>0$ such that 
\begin{align}\label{C3_18}
I_{\lambda,p}(u)<0\;\;\textrm{ for all } u\in \textrm{span}\{\varphi_1,\varphi\},\ \|u\|_{X_p^s}\geq\bar{R}.
\end{align}
By applying Lemma \ref{C2PE_51} for  $\rho\varphi\in\textrm{span}\{\varphi_1,\varphi\}$, we obtain
$$I_{\lambda,p}(\rho\varphi)\geq \beta>0,$$
proving that $\bar{R}>\rho$.

We now consider 
$$Q=\{u=w+t\varphi,\ w\in\textrm{span}\{\varphi_1\}\cap B_{\bar{R}},\; 0\leq t\leq\bar{R}\}$$
and consider the border $\partial Q=\displaystyle\bigcup_{i=1}^{3}\Gamma_i$ with
 \begin{enumerate}
\item $\Gamma_1=\overline{B}_{\bar{R}}(0)\cap\textrm{span}\{\varphi_1\}$,   
\item $\Gamma_2=\{u\in X_p^s\,:\, u=w+\bar{R}\varphi, w\in B_{\bar{R}}(0)\cap \textrm{span}\{\varphi_1\}\}$, 
\item $\Gamma_3=\{u\in X_p^s\,:\, u=w+r\varphi,\; w\in\textrm{span}\{\varphi_1\},\; \|w\|_{X_p^s}=\bar{R},\ 0\leq r\leq \bar{R}\}$.
\end{enumerate}

We have $I_{\lambda,p} \big |_{\Gamma_i}\leq \eta(\lambda)$, for $i=1, 2, 3$.

In fact, this follows from Lemma \ref{C3PE_9} if $u\in\Gamma_1\subset\textrm{span}\{\varphi_1\}$. 
However, if $u\in \Gamma_2$ or $u\in \Gamma_3$, then it is a consequence of  \eqref{C3_18}.

By the Linking theorem, there exists a weak solution $u_\lambda\in X_p^s$ of the problem \eqref{principal2} such that $$0<\eta(\lambda)<\beta\leq I_{\lambda,p}(u_\lambda)=C_{\lambda}.$$ Observe that $u_\lambda\neq 0$, since $I_{\lambda,p}(0)=0$. 

In order to show that this third solution  is different from the positive and negative solutions obtained before,
consider $g_0^{+}:[0,1]\to X_p^s$ given by $g_0^{+}(t)=t(t_0\varphi_1)$, with $t_0$ defined in Lemma \ref{C2PE_34}.
We have  
$$g_{0}^+\in\{g\in C([0,1],X_p^s)\, :\, g(0)=0,\ g(1)=t_0\varphi_1\}.$$

It follows from Lemma \ref{C3PE_9} that
\begin{equation}\label{curvaPE}
I_{\lambda,p}^+(g_0^+(t))=I_{\lambda,p}(g_0^+(t)) \leq\eta(\lambda),\,\textrm{ for all } t\in[0,1].
\end{equation}

To conclude that the solution established by the Linking theorem is different from the positive and negative solutions obtained before, we apply the result analogous of Lemma \ref{C2PE_34}, valid for solutions with negative energy and define $g_0^-\in\Gamma^-$ satisfying an estimate analogous to \eqref{curvaPE}.
$\hfill\Box$\end{proof}

\textit{Proof of Theorem \ref{3}.} To conclude its proof we observe that, if $f$ is odd, then $I_{\lambda,p}$ is even. Now, the existence of infinite many solutions follows by applying the symmetric version of the Mountain Pass Theorem, see \cite[Theorem 9,12]{rabinowitz}.
$\hfill\Box$

\section{Proof of Theorem \ref{4}}
\subsection{Positive and negative solutions for the functional $I_{\lambda,p}$}
\begin{lemma}\label{C4PE_3} Suppose that $f$ satisfies the hypotheses of the critical exponential growth case. Then the functional $$I_{\lambda,p}(u)=\displaystyle\frac{1}{p} \|u\|_{X_p^s}^p+\frac{\lambda}{q}\int_{\Omega}\vert u\vert^q\dd x-\frac{a}{p}\int_{\Omega}\vert u\vert^p\dd x-\int_{\Omega} F(u)\dd x,$$ satisfies the $(PS)$-condition at any level  $c<\displaystyle\frac{s}{N}\left(\frac{\alpha_{s,N}^*}{\alpha_0}\right)^{\frac{N-s}{s}}$. 
\end{lemma}
\begin{proof}
For $c<\frac{s}{N}\left(\frac{\alpha_{s,N}^*}{\alpha_0}\right)^{\frac{N-s}{s}}$, let $(u_n)$ be a $(PS)_c$ sequence in $X_p^s$. Lemma \ref{C2_9} guarantees that  $(u_n)$ is bounded. Therefore, passing to a subsequence, we can suppose that
$$\begin{cases}
u_n\rightharpoonup u\ \textrm{ in }\, X_p^s,\\
u_n\to u \textrm{ in }\,  L^{q}(\Omega)\textrm{ for all }\, q\geq 1\\
u_n(x)\to u(x)\textrm{ a.e. in }\, \Omega.
\end{cases}
$$

Since $\lambda>0$, $(\|I_{\lambda,p}'(u_n)\|_{(X_p^s)^*})$ and $(I_{\lambda,p}(u_n))$ are bounded sequences in $\mathbb{R}$. Therefore,
\begin{align*}
\int_{\Omega} f(u_n)u_n\dd x
&\leq \|u_n\|^p_{X_p^s}+\lambda\int_{\Omega}\vert u_n\vert^q\dd x+a\int_{\Omega}\vert u_n\vert^p\dd x+\|I_{\lambda,p}'(u_n)\|\|u_n\|_{X_p^s}\leq C_1\\
\int_{\Omega} F(u_n)\dd x&\leq\frac{\|u_n\|^p_{X_p^s}}{p}+\frac{\lambda}{q}\int_{\Omega}\vert u_n\vert^q\dd x+\frac{a}{p}\int_{\Omega}\vert u_n\vert^p\dd x+\vert I_{\lambda,p}(u_n)\vert\leq C_2.
\end{align*}

Because $f$ satisfies $(f_{5,p})$, it also satisfies 
\begin{equation}\label{ARp}
pF(t)\leq tf(t)\quad \textrm{ for all } t\neq 0.\end{equation} 
Therefore, there exists $C>0$ such that 
\begin{equation*}
\max\left\lbrace \|u_n\|_{X_p^s}^2, \int_{\Omega} f(u_n)u_n\dd x, \int_{\Omega} F(u_n)\dd x\right\rbrace\leq C.
\end{equation*}

It follows from \eqref{apriori} that  
$f(u_n), f(u)\in L^1(\Omega)$. Since $u_n\to u$ in $L^1(\Omega)$ and $\int_{\Omega} f(u_n)u_n\dd x\leq C$, we conclude that  
\begin{equation}\label{C4PE_5}
f(u_n)\to f(u) \mbox { in } L^1(\Omega)
\end{equation}
by applying \cite[Lema 2.1]{deFigueiredo}. Thus,  \eqref{C4PE_5} and $(f_{6,p})$ allow us to conclude that
\begin{equation}\label{C4PE_6}
F(u_n)\to F(u)\;\textrm{ in }\,L^1(\Omega)
\end{equation} 
and 
\begin{equation*}
\frac{\|u_n\|_{X_p^s}^p}{p}\to c-\frac{\lambda}{q}\int_{\Omega}\vert u\vert^q\dd x+\frac{a}{p}\int_{\Omega}\vert u\vert^p\dd x+\int_{\Omega} F(u)\dd x.
\end{equation*} 

Since $I'_{\lambda}(u_n)\to 0$ in $(X_p^s)^{*}$, it follows that 
\begin{equation}\label{C4PE_8}
\int_{\Omega} f(u_n)u_n\dd x
\to pc+\lambda\left(1-\frac{p}{q}\right)\int_{\Omega}\vert u\vert^q\dd x+p\int_{\Omega} F(u)\dd x.
\end{equation}

A new application of \eqref{ARp} yields \[\displaystyle\int_{\Omega} f(u_n)u_n\dd x-p\int_{\Omega}F(u_n)\dd x\geq 0.\]
We conclude from \eqref{C4PE_8} that 
\begin{equation*}
pc\geq \lambda\left(\frac{p}{q}-1\right)\int_{\Omega}\vert u\vert^q,
\end{equation*}
thus showing that $c\geq 0$. Now, standard arguments show that
$\langle I_{\lambda,p}'(u),v\rangle=0$ for all $v\in X_p^s$.  

Thus, \eqref{ARp} yields  
\begin{align*}
I_{\lambda,p}(u)
&\geq \frac{1}{p}\left(\|u\|_{X_p^s}^p+\frac{p\lambda}{q}\int_{\Omega}\vert u\vert^q-a\int_{\Omega}\vert u\vert^p\dd x-\int_{\Omega} f(u)u\dd x\right)\\
&> \frac{1}{p}\left(\|u\|_{X_p^s}^p+\lambda\int_{\Omega}\vert u\vert^q-a\int_{\Omega}\vert u\vert^2\dd x-\int_{\Omega} f(u)u\dd x\right)\\
&=\frac{1}{p}\langle I'_{\lambda,p}(u),u\rangle
=0
\end{align*}
proving that $I_{\lambda,p}(u)>0$, since $I_{\lambda,p}(0)=0$, and also that $c>0$ and $u\neq 0$.

To prove that $u_n\to u$ in $X_p^s$, it suffices to show that $I_{\lambda,p}(u)=c$, since this yields $ \|u_n\|_{X_p^s}\to\|u\|_{X_p^s}$.

In fact, it follows from \eqref{C4PE_6} that 
\begin{align*}
I_{\lambda,p}(u)
&\leq \displaystyle{\liminf_{n\to\infty}}\left(\frac{1}{p}\|u_n\|_{X_p^s}^p\right)+\displaystyle{\liminf_{n\to\infty}}\left(\frac{\lambda}{q}\int_{\Omega}\vert u_n\vert^q-\frac{a}{p}\int_{\Omega}\vert u_n\vert^p\dd x-\int_{\Omega} F(u_n)\dd x\right)\\
&\leq \displaystyle{\liminf_{n\to\infty}}\left(\frac{1}{p}\|u_n\|_{X_p^s}^p+\frac{\lambda}{q}\int_{\Omega}\vert u_n\vert^q-\frac{a}{p}\int_{\Omega}\vert u_n\vert^p\dd x-\int_{\Omega} F(u_n)\dd x\right) \\
&=\displaystyle{\liminf_{n\to\infty}}\;I_{\lambda,p}(u_n)=c.
\end{align*}

If $I_{\lambda,p}(u)<c$, then we would have
\begin{align}\label{C4PE_14}
\displaystyle{\lim_{n\to\infty}}\|u_n\|_{X_p^s}^p&=p\left(c-\frac{\lambda}{q}\int_{\Omega}\vert u\vert^q+\frac{a}{p}\int_{\Omega}\vert u\vert^p\dd x+\int_{\Omega} F(u)\dd x\right)\nonumber\\
&>p\left(I_{\lambda,p}(u)-\frac{\lambda}{q}\int_{\Omega}\vert u\vert^q+\frac{a}{p}\int_{\Omega}\vert u\vert^2\dd x+\int_{\Omega} F(u)\dd x\right)\nonumber\\
&=\|u\|_{X_p^s}^p.
\end{align}

By defining  $v_n=\dfrac{u_n}{\|u_n\|_{X_p^s}}$ and $v=\dfrac{u}{c_0}$
where
 $$c_0=\left(pc-\frac{p\lambda}{q}\int_{\Omega}\vert u\vert^q+a\int_{\Omega}\vert u\vert^p\dd x+p\int_{\Omega} F(u)\dd x\right)^{-1/p}>0,$$
\eqref{C4PE_14} would then imply that
$$\|v\|_{X_p^s}=\frac{\|u\|_{X_p^s}}{c_0}<\frac{\|u\|_{X_p^s} }{\|u\|_{X_p^s}}=1.$$ 

Since we can conclude that
$v_n\rightharpoonup v$ in $X_p^s$, by choosing $\alpha>\alpha_0$ so that
$$pr^{\frac{N-s}{s}}\alpha^{\frac{N-s}{s}}<\frac{\displaystyle\left(\alpha_{s,N}^*\right)^{\frac{N-s}{s}} }{c},$$
$I_{\lambda,p}(u)>0$ and \eqref{C4PE_14} would then imply
\begin{align*}
\displaystyle{\lim_{n\to\infty}}r^{\frac{N-s}{s}}\alpha^{\frac{N-s}{s}}\|u_n\|_{X_p^s}^p&=r^{\frac{N-s}{s}}\alpha^{\frac{N-s}{s}}c_0^p\nonumber\\
&<\displaystyle\left(\alpha_{s,N}^*\right)^{\frac{N-s}{s}}\left(\frac{c_0^p}{p(c-I_{\lambda,p}(u))}\right).
\end{align*}

It is not difficult to show that
\begin{equation*}
\frac{c_0^p}{p(c-I_{\lambda)}(u)}=\frac{1}{1-\|v\|_{X_p^s}^p}.
\end{equation*}

Thus,  
$$b=\displaystyle{\lim_{n\to\infty}}r^{\frac{N-s}{s}}\alpha^{\frac{N-s}{s}}\|u_n\|_{X_p^s}^p<\frac{\displaystyle\left(\alpha_{s,N}^*\right)^{\frac{N-s}{s}}}{1-\|v\|_{X_p^s}^p}.$$
So, for $\epsilon>0$ small enough, we have $$r\alpha\|u_n\|_{X_p^s}^\frac{N}{N-s}<\epsilon+b<\frac{\alpha_{s,N}^*}{(1-\|v\|_{X_p^s}^p)^{\frac{s}{N-s}}}$$ for all $n\in\mathbb{N}$ large enough. Thus, there exist $1<\mu<\frac{1}{(1-\|v\|_{X_p^s}^p)^{\frac{s}{N-s}}}$ and $0<\gamma<\alpha^{*}_{s,N}$ such that $$r\alpha\|u_n\|_{X_p^s}^\frac{N}{N-s}<\gamma \mu<\frac{\alpha^{*}_{s,N}}{(1-\|v\|_{X_p^s}^p)^{\frac{s}{N-s}}}.$$

But \eqref{apriori} implies
\begin{align*}
\int_{\Omega}\vert f(u_n)\vert^{r}\dd x &\leq  C\int_{\Omega} \exp(r\alpha\vert u_n\vert^\frac{N}{N-s})\dd x
\leq C\int_{\Omega}\exp\left(\gamma \mu \vert v_n\vert^\frac{N}{N-s}\right) \dd x.
\end{align*}
Our choice of $\mu$ and $\gamma$ guarantees that the sequence  
$\left(\exp\left(\gamma \vert v_n\vert^\frac{N}{N-s
}\right)\right)$ is bounded in $L^{\mu}(\Omega)$. Thus, there exists $K>0$ such that 
$$\int_{\Omega}\vert f(u_n)\vert^{r}\dd x\leq K,$$ 
proving that $(f(u_n))$ is bounded in  $L^r(\Omega)$ for some $r>1$. 

By applying the Brezis-Lieb lemma, we conclude that
$f(u_n)\rightharpoonup f(u)$ in $L^r(0,1)$ and, since $u_n\to u$ in $L^{r'}(0,1)$, 
we conclude that
\begin{equation*}
\displaystyle{\lim_{n\to\infty}}\int_{\Omega} f(u_n)u_n\dd x=\int_{\Omega} f(u)u\dd x.
\end{equation*} 

Thus,
\begin{align*}
\displaystyle{\lim_{n\to\infty}}\|u_n\|_{X_p^s}^p&=  
\displaystyle{\lim_{n\to\infty}}\left(\langle I'_{\lambda,p}(u_n),u_n\rangle-\frac{\lambda}{q}\int_{\Omega}\vert u_n\vert^q+a\int_{\Omega}\vert u_n\vert^p\dd x+\int_{\Omega} f(u_n)u_n\dd x\right)\\
&=-\frac{\lambda}{q}\int_{\Omega}\vert u\vert^q+a\int_{\Omega}\vert u\vert^p\dd x+\int_{\Omega} f(u)u\dd x\\
&=\|u\|_{X_p^s}^p-\langle I'_{\lambda,p}(u),u\rangle=\|u\|_{X_p^s}^p.
\end{align*} 
we have reached a contradiction. Therefore, $I_{\lambda,p}(u)=c$.
$\hfill\Box$\end{proof}
\begin{remark}
The same result is valid for the functionals $I_{\lambda,p}^{+}$ and $I^{-}_{\lambda}$.
\end{remark}

\begin{proposition}\label{sp}
Suppose that $a\geq \lambda_1$ and $f$ satisfies $(f_{1,p}), (f'_{2,p})$, $(f_{3,p})$ and $(f_{5,p})-(f_{7,p})$. Then, in the case of critical exponential growth, problem \eqref{principal2} has at least one positive solution for all $\lambda>0$ small enough. 
\end{proposition}
\begin{proof}
As in the subcritical growth case, the  functional $I_{\lambda,p}^+$ satisfies the geometric hypotheses of the Mountain Pass Theorem.

We will show that $I_{\lambda,p}^+$ satisfies the $(PS)$ condition at level $C_\lambda^+$, given by
\[C_\lambda^+=\displaystyle\inf_{g\in\Gamma^+}\displaystyle\max_{u\in g([0,1])} {I_{\lambda,p}^+(u)},\] 
where $\Gamma^+=\{g\in C([0,1],X_p^s)\,:\, g(0)=0,\, g(1)=t_0\varphi_1\}$, with $t_0$ given in Lemma \ref{C2PE_34}. 

Observe that $$\displaystyle\max_{t\in[0,1]}I_{\lambda,p}^+(g(t))\geq I_{\lambda,p}^+(g(0))=I_{\lambda,p}^+(0)=0,\;\forall\; g\in\Gamma^+,$$
thus implying that $\displaystyle\max_{t\in[0,1]}I_{\lambda,p}^+(g(t))\geq 0,\ \forall\; g\in\Gamma^+$.

It follows that $$0\leq \displaystyle\inf_{g\in\Gamma^+}\displaystyle\max_{u\in g([0,1])} {I_{\lambda,p}^+(u)}=C_\lambda^+<\infty.$$ 

As in the proof of Lemma \ref{C4PE_3}, we obtain that $I_{\lambda,p}^+$ satisfies the $(PS)_c$ condition for all $\lambda>0$, where $c<\frac{s}{N}\left(\frac{\alpha_{s,N}^*}{\alpha_0}\right)^{\frac{N-s}{s}}$. 

We will show that  
$C_\lambda^+<\frac{s}{N}\left(\frac{\alpha_{s,N}^*}{\alpha_0}\right)^{\frac{N-s}{s}}$, if $\lambda>0$ is small enough.

In fact, by defining $g_0^{+}:[0,1]\to X_p^s$ by
$g_0^{+}(t)=t(t_0\varphi_1)$, the result follows by applying Lemma \ref{C3PE_9}: 
$$C_{\lambda}^+\leq I_{\lambda,p}^+(g_0^+)<\eta(\lambda)<\displaystyle\frac{s}{N}\left(\frac{\alpha_{s,N}^*}{\alpha_0}\right)^{\frac{N-s}{s}},$$ if $\lambda>0$ is small enough.
$\hfill\Box$\end{proof}

The proof of existence of a negative solution is analogous to that of Proposition \ref{sp}. 
\begin{proposition}
Suppose that $a\geq \lambda_1$ and that  $f$ satisfies $(f_{1,p}), (f'_{2,p}),(f_{3,p})$ and $(f_{5,p})-(f_{7,p})$. Then, in the case of critical exponential growth, problem \eqref{principal2} has at least one negative solution for all $\lambda>0$ small enough. 
\end{proposition}
 \subsection{A third solution}
\begin{proposition}
Suppose that $f$ satisfies $(f_{1,p}), (f'_{2,p}), (f_{3,p})$, and $(f_{5,p})-(f_{7,p})$. If $\lambda_{1}\leq a<\lambda^{*}$ then, for all $\lambda>0$ small enough, problem \eqref{principal2} has at least a third solution in case of the critical exponential growth.
\end{proposition}
\begin{proof}
According to Lemmas  \ref{C2PE_51} and \ref{C2PE_52}, the functional $I_{\lambda,p}$ satisfies the geometry of the Linking Theorem.

We maintain the notation introduced in Section \ref{subcritico}, with $X_p^s=W\oplus\textup{span}\{\varphi_1\}$ and $Q=(B_R\cap \textrm{span}\{\varphi_1\})\oplus([0,R\varphi])$ for $\varphi\in W$. So, it suffices to prove that 
\begin{enumerate}
\item[$(iii)$] $\displaystyle\sup_{u\in Q}I_{\lambda,p}(u)<\displaystyle\frac{s}{N}\left(\frac{\alpha_{s,N}^*}{\alpha_0}\right)^{\frac{N-s}{s}}$.
\end{enumerate}

We claim that $(f_{7,p})$ implies
\begin{equation}\label{C4PE_19}
\max_{u\in\bar{Q}}\;I_{\lambda,p}(u)<\displaystyle\frac{s}{N}\left(\frac{\alpha_{s,N}^*}{\alpha_0}\right)^{\frac{N-s}{s}},
\end{equation}
for $\lambda>0$ small enough.

So, we write $I_{\lambda,p}$ in the form $$I_{\lambda,p}(u)=J(u)+\dfrac{\lambda}{q}\displaystyle\int_{\Omega}\vert u\vert^q\dd x,$$ with $J(u)=\dfrac{1}{p}\|u\|_{X_p^s}^p-\dfrac{a}{p}\int_{\Omega}\vert u\vert^p\dd x-\int_{\Omega}F(u)\dd x$.

In order to prove ($iii$), it is enough to verify that
\begin{equation}
\label{ineqcrit}\sup_{u\in\bar{Q}}J(u)<\displaystyle\frac{s}{N}\left(\frac{\alpha_{s,N}^*}{\alpha_0}\right)^{\frac{N-s}{s}}
\end{equation}
or, what is the same, that for $\lambda>0$ small enough, we have
$$\displaystyle\sup_{u\in\overline{ Q}}I_{\lambda,p}(u)\leq\displaystyle\sup_{u\in\overline{ Q}}J(u)+\dfrac{\lambda}{q}\displaystyle\sup_{u\in\bar{Q}}{\vert u\vert_{L^q(\Omega)}^q}<\displaystyle\frac{s}{N}\left(\frac{\alpha_{s,N}^*}{\alpha_0}\right)^{\frac{N-s}{s}},$$ thus showing ($iii$).

In order to prove \eqref{ineqcrit}, we will show that 
$$\displaystyle\sup_{u\in\bar{Q}}J(u)<\displaystyle\frac{s}{N}\left(\frac{\alpha_{s,N}^*}{\alpha_0}\right)^{\frac{N-s}{s}}.$$ 
 
Consider $\mathbb{F}=\textrm{span}\{ \varphi_1,\varphi\}$. We have 
\begin{align*}
\displaystyle\sup_{u\in\bar{Q}}J(u)\leq \displaystyle\max_{u\in\mathbb{F}}J(u)
=\displaystyle\max_{u\in\mathbb{F}, t\neq 0}J\left(\vert t\vert\frac{u}{\vert t\vert}\right)
=\displaystyle\max_{u\in\mathbb{F}, t> 0}J(tu)\leq\displaystyle\max_{u\in\mathbb{F}, t\geq 0}J(tu). 
\end{align*} 

But
\begin{align*}
J(tu)&=\displaystyle\frac{t^p}{p}\|u\|_{X_p^s}^p-\frac{a}{p}t^p\int_{\Omega}\vert u\vert^p\dd x-\int_\Omega F(tu )\dd x\nonumber\\
&\leq \displaystyle\frac{t^p}{p}\|u\|_{X_p^s}^p-\int_{\Omega} F(tu)\dd x.
\end{align*}

Define $\eta\colon[0,+\infty)\to\mathbb{R}$ by 
$$\eta(t)=\displaystyle\frac{t^p}{p}\|u\|_{X_p^s}^p-\int_{\Omega} F(tu)\dd x.$$

Since all norms in $\mathbb{F}$ are equivalent, it follows from $(f_{7,p})$ the existence of $C>0$ such that
\begin{align*}
\int_{\Omega} F(tu)\dd x \geq& \dfrac{C_r}{r}\int_{\Omega}t^r\vert u\vert^r\dd x
=\dfrac{C_r}{r}t^r\|u\|_{L^r(\Omega)}^r\geq C\dfrac{C_r}{r} t^r\|u\|_{X_p^s}^{r}.
\end{align*}

Thus,
\begin{equation}\label{C4_21.00}
\eta(t)\leq \displaystyle\frac{t^p}{p}\|u\|_{X_p^s}^p-C\dfrac{C_r}{r} t^r\|u\|_{X_p^s}^{r}
\leq\displaystyle{\max_{t\geq 0}}\left(\displaystyle\frac{t^p}{p}\|u\|_{X_p^s}^p -C\dfrac{C_r}{r} t^r\|u\|_{X_p^s}^{r} \right).
\end{equation}

If we denote  $g(t):=\displaystyle\frac{t^p}{p}\|u\|_{X_p^s}^p -C\dfrac{C_r}{r} \|u\|_{X_p^s}^{r}  t^r$, it is easy to prove that $g$ attains its (global) maximum at 
$t_1=\displaystyle\left(\frac{\|u\|_{X}^{p-r}}{CC_r}\right)^\frac{1}{r-p}$ and 
$$\displaystyle{\max_{t\geq 0}}\; g(t)=g(t_1)=\left(\frac{1}{CC_r}\right)^{\frac{p}{r-p}} \left(\frac{r-p}{pr}\right)> 0.$$
 
Thus, it follows from $(f_{7,p})$ that
\begin{align*}
\max_{t\geq 0}g(t)
&<\frac{1}{2} \frac{s}{N}\left(\frac{\alpha_{s,N}^*}{\alpha_0}\right)^{\frac{N-s}{s}}\left(\frac{pr}{r-p}\right)\left(\frac{r-p}{pr}\right)
=\frac{1}{2} \frac{s}{N}\left(\frac{\alpha_{s,N}^*}{\alpha_0}\right)^{\frac{N-s}{s}}.
\end{align*}

Therefore, \eqref{C4_21.00} yields $\eta(t)<\frac{1}{2}\dfrac{s}{N}\left(\frac{\alpha_{s,N}^*}{\alpha_0}\right)^{\frac{N-s}{s}}$ and 
$$\displaystyle{\max_{u\in\mathbb{F},\; t\geq 0}}J(tu)\leq \displaystyle{\max_{t\geq 0}}\,\eta(t)\leq \frac{1}{2}  \frac{s}{N}\left(\frac{\alpha_{s,N}^*}{\alpha_0}\right)^{\frac{N-s}{s}}<\frac{s}{N}\left(\frac{\alpha_{s,N}^*}{\alpha_0}\right)^{\frac{N-s}{s}},$$ 
and the proof of our claim is complete.

We will now show that, for $\lambda>0$ small enough, the functional $I_{\lambda,p}$ satisfies the  $(PS)$-condition at the level $$C_\lambda=\inf_{h\in\Gamma}\sup_{u\in\bar{ Q}}I_{\lambda,p}(h(u)),$$
where $\Gamma=\{h\in C(\bar{Q},X_p^s)\;;\;h=id\textrm{ in }\partial Q\}.$ 

In fact, \eqref{C4PE_19} implies that, for $\lambda>0$ small enough, we have   
$$\displaystyle\sup_{u\in\bar{Q}}I_{\lambda,p}(u)< \frac{s}{N}\left(\frac{\alpha_{s,N}^*}{\alpha_0}\right)^{\frac{N-s}{s}}.$$

Thus,
$$\displaystyle\inf_{h\in\Gamma}\displaystyle\sup_{u\in\bar{Q}}I_{\lambda,p}(h(u))\leq\displaystyle\sup_{u\in\bar{Q}}I_{\lambda,p}(u)<\frac{s}{N}\left(\frac{\alpha_{s,N}^*}{\alpha_0}\right)^{\frac{N-s}{s}}$$
and the $(PS)_{C_\lambda}$-condition is consequence of Lemma \ref{C4PE_3}.

It follows from the Linking Theorem that  $C_\lambda=\displaystyle\inf_{h\in\Gamma}\displaystyle\sup_{u\in\bar{Q}}I_{\lambda,p}(h(u))$ 
is a critical value for $I_{\lambda,p}$, with $C_\lambda\geq\beta$. Therefore, there exists $u_\lambda\in X_p^s$ weak solution of \eqref{principal2} satisfying $0<\beta\leq I_{\lambda,p}(u_\lambda)$, what implies that $u_\lambda\neq 0$.

As in the proof of Proposition \ref{proplinking}, we prove that this solution is different from the positive and negative solutions already obtained. 
$\hfill\Box$\end{proof} 

Observe that we also conclude the proof of Theorem \ref{4} by the same reasoning given in the proof of Theorem \ref{3}.
\section{Proof of theorems \ref{5} and \ref{6}} 
The proofs of Theorems \ref{5} and \ref{6} follow the same arguments presented in the proofs of Theorems  \ref{3} and \ref{4}, respectively; in order to find a third solution by applying the Linking Theorem we consider the decomposition
\[X=V_k\oplus W_k,\] where $V_k=\textrm{span}\{\varphi_1,\ldots,\varphi_k\}$ is the subspace generated by the autofunctions of $(-\Delta)^{1/2}$ corresponding to the eigenvalues $\lambda_1,\ldots,\lambda_k$, e $W_k=V_k^\perp$. 

\vspace{0.5cm}
\noindent
{\bf Acknowledgment}  The authors would like to thank A. Iannizzotto for fruitful discussions with respect to Theorem \ref{0}.

\end{document}